\documentclass[12pt,reqno]{amsart}

\usepackage{amsmath,amssymb,amsfonts,latexsym,amscd,enumitem,amsthm}
\usepackage{stmaryrd}
\usepackage{graphicx}
\usepackage[all,cmtip]{xy}
\usepackage{mathrsfs}

\usepackage[pagebackref=false,colorlinks=true,linkcolor=red,citecolor=red]{hyperref}
\hypersetup{backref,pdfpagemode=FullScreen,hypertex,hyperindex=true}

\usepackage{todonotes}

\makeatletter
\def\@seccntformat#1{%
	\protect\textup{%
		\protect\@secnumfont
		\expandafter\protect\csname format#1\endcsname 
		\csname the#1\endcsname
		\protect\@secnumpunct
	}%
}

\makeatother

\theoremstyle{plain}
\newtheorem{theorem}{Theorem}[section]
\newtheorem{lemma}[theorem]{Lemma}
\newtheorem{proposition}[theorem]{Proposition}
\newtheorem{corollary}[theorem]{Corollary}

\theoremstyle{definition}
\newtheorem{definition}[theorem]{Definition}

\newtheorem{remark}[theorem]{Remark}

\newtheorem{example}[theorem]{Example}

\def\open{{\mathfrak o}}
\def\close{{\mathfrak c}}

\newcommand{\tbigwedge}{\mathop{\textstyle \bigwedge}}
\newcommand{\tbigcap}{\mathop{\textstyle \bigcap}}
\newcommand{\tbigcup}{\mathop{\textstyle \bigcup}}
\newcommand{\tbigvee}{\mathop{\textstyle \bigvee}}

\def\OS{\text{\sf O}}
\def\pprec{{\mathop{\hskip1pt\prec\hskip-2pt \prec\hskip2pt}}}

\newcommand{\ur}{{{\rlap{$\ $}\hbox{$\uparrow$}}}}

\newcommand{\QQ}{\mathbb{Q}}

\newcommand{\cb}{\prec\!\!\prec} 
\newcommand{\Ss}{\mathsf{S}}
\newcommand{\Sz}{\mathsf{ZS}}
\newcommand{\Sc}{\mathsf{CoZS}}
\newcommand{\coz}{\mathsf{Coz}}

\newcommand{\z}{\mathsf{z}}
\newcommand{\co}{\mathsf{coz}}
\newcommand{\ro}{\mathsf{reg}_\open}
\newcommand{\rc}{\mathsf{reg}_\close}
\newcommand{\ddown}{\rotatebox[origin=c]{90}{$\twoheadleftarrow$}}

\begin{document}
	
	\title[Generalizing $\beta$- and $\lambda$-maps]{Generalizing $\beta$- and $\lambda$-maps}
	
	\author[A. B. Avilez]{Ana Bel\'{e}n Avilez}
	
	\address{\hspace*{-\parindent}University of the Western Cape, Department of Mathematics and Applied Mathematics,   \newline 7535 Bellville, South Africa \newline {\it Email addresses}: {\tt anabelenavilez@gmail.com}}
	
	
	\date{July 23, 2024}
	\begin{abstract}
		We generalize the notions of $\beta$- and $\lambda$-maps to general selections of sublocales, obtaining different classes of localic maps. These new classes of maps are used to characterize almost normality, extremal disconnectedness, $F$-frames, $Oz$-frames, among others types of locales, in a manner akin to the characterization of normal locales via $\beta$-maps. 
		As a byproduct we obtain a characterization of localic maps that preserve the completely below relation (that is, the right adjoints of assertive frame homomorphisms).
	\end{abstract}
	\maketitle
	\section*{Introduction}
	
	Motivated by the idea of characterizing different types of locales using only localic maps, we generalize the notions of $\beta$- and $\lambda$-maps obtaining a variety of classes of maps and offering several characterizations of frames only on conditions on the localic maps. 
	
	In classical topology, a continuous map $k\colon X\to Y$ is an $N$-map (resp. $WN$-map) if 
	\[\mathsf{cl}_{\beta X}k^{-1}[F] = (k^\beta)^{-1} [\mathsf{cl}_{\beta Y} F]\]
	for every closed subset (resp. zero subset) $F$ of $Y$, where $k^\beta$ is the extension of $k$ that maps $\beta X$ to $\beta Y$. 
	These maps were used in \cite{Woods} to give characterizations of normal and $\delta$-normally separated spaces. In \cite{Remote,Socle} the corresponding notions in the point-free setting, $N$- and $WN$-homomorphisms, were firstly introduced. Later in \cite{open,closed} the authors, not only work with the Stone-\v{C}ech compactification, but similarly define $\lambda$-maps (the analogous of $\beta$-maps or $N$-homomorphisms) using the Lidel\"{o}fication of a frame. 
	
	In this paper, we continue the study that started in \cite{maps} where the author generalizes the notion of $\beta$-maps to selections, in order to obtain a general result that characterizes different types of frames. So far these maps were only defined using selections of closed sublocales. Following the reasoning of \cite{parallel} we explore the dual results, when taking selections of open sublocales. Moreover, the same study will be developed for generalized $\lambda$-maps. 
	
	The main purpose of this paper, different from the point of view in \cite{open,closed} which was studying the liftings of frame homomorphisms under the Stone-\v{C}ech compactification and Lindel\"{o}fication, we intend to provide new classes of localic maps and explore the imporatance of these maps in its own right (discarding, in a way, the relation they have with the Stone-\v{C}ech compactification and Lindel\"{o}fication). This paper should be thought of as an attempt to continue the study of different classes of localic maps. 
	
	Here is a brief outlay of the paper. 
	In Section 1 we give some general background and notation. In the following section we introduce a general definition of a $\mathfrak{S}\gamma$-map where $\mathfrak{S}$ stands for a selection of sublocales and $\gamma$ for a special kind of functor. In Section 3 we work with $\mathfrak{S}\beta$-maps, where $\beta$ denotes the Stone-\v{C}ech compactification. We first recall and extend some of the known results, then we take a closer look at $\mathfrak{S}\beta$-maps when the selection function selects classes of open sublocales. This approach is new (as it was never done classically or point-freely) and allows us to obtain dual results as the ones already known and characterize the localic maps that preserve the completely rather below relation.
	Doing a similar study as in the previous section, we introduce and characterize $\mathfrak{S}\lambda$-maps in Section 4, where $\lambda$ stands for the Lindel\"{o}fication of a frame. We then generalize the definition of $\mathfrak{S}\lambda$-maps obtaining a whole range of classes of localic maps and two characterizing theorems for special types of frames. Finally, in Section 5 we look at special cases when embeddings of sublocales are $\mathfrak{S}\beta$- or $\mathfrak{S}\lambda$-maps.

	\section{Notation and Terminology}
	
	\subsection{The categories of frames and locales}
	
	Our notation and terminology for
	frames and locales is that of \cite{PP12}. We recall here some of the basic notions involved. A {\it frame} $L$ is a complete lattice in which
	\begin{equation}\label{frmdist}
		a\wedge \tbigvee S=\tbigvee \{a\wedge b\mid b\in S\}\quad \mbox{ for any }a\in L\mbox{ and }S\subseteq L.
	\end{equation}
	A {\em frame homomorphism} preserves all joins (in particular, the bottom element $0$ of the lattice) and all finite meets (in particular, the top element 1).
	
	In any frame $L$, the mappings
	$(x\mapsto (a\wedge x))\colon L\to L$ preserve suprema and hence they have right Galois adjoints $(y\mapsto (a\to y))\colon L\to L$, satisfying
	\[
	a\wedge x\leq y \quad\text{iff}\quad x\leq a\to y
	\]
	and making $L$ a complete Heyting algebra.
	The {\em pseudocomplement} of $a\in L$ is the element
	$a^*=a\to 0=\tbigvee\{x\mid x\wedge a=0\}$ and we write $L^*:=\left\{a^*\mid a\in L\right\}$. A {\em regular element} of $L$ is an element of the form $a^*$ for some $a$ (equivalently, an element $a$ such that $a^{**}=a$).
	
	The {\em rather below relation} $\prec$ in  $L$  is defined by $b\prec a$ iff $b^*\vee a=1$. A frame is {\em regular} if $a=\tbigvee\{b\in L\mid b\prec a\}$ for every $a\in L$, and it is {\em normal} if for any $a,b\in L$ such that $a\vee b=1$,  there is a pair $u,v\in L$ such that $u\wedge v=0$ and $a\vee u=b\vee v=1$.
	
	The category of locales and localic maps is the dual category of frames and frame homomorphisms. Thus a {\it locale} is a frame and {\em localic maps} can be represented by the
	(uniquely defined) right adjoints $f=h_*\colon L\to M$ of frame homomorphisms  $h\colon M\to L$. They are precisely the meet preserving maps
	$f\colon L\to M$ such that $f(h(a)\to b)=a\to f(b)$ and $f(a)=1\Rightarrow a=1$.\\
	Throughout the paper we will always use $f$ to refer to localic maps and $h$ will always denote the frame homomorphism left adjoint of $f$.

	\subsection{The coframes of sublocales of a locale}
	A {\em sublocale} of a locale $L$ is a subset $S\subseteq L$ closed under arbitrary meets such that
	\[\forall x\in L\ \ \forall s\in S \ \ (x\to s\in S).\]
	They are precisely the subsets of $L$ for which the embedding $j\colon S\hookrightarrow L$ is a localic map. The system $\Ss(L)$ of all sublocales of $L$, partially ordered by inclusion, is a {\em coframe}, that is, its dual lattice is a frame \cite[Theorem~III.3.2.1]{PP12}.  Infima and suprema are given by
	\[
	\tbigwedge_{i\in J}S_i=\tbigcap_{i\in J}S_i, \quad \tbigvee_{i\in J}S_i=\{\tbigwedge M\mid M\subseteq\tbigcup_{i\in J} S_i\}.
	\]
	The least element is the {\em void sublocale} $\OS=\{1\}$ and the greatest element is the entire locale $L$. Since $\Ss(L)$ is a coframe, there is a co-Heyting operator $S\smallsetminus T$ given by the formula $\bigcap\left\{A\mid S\subseteq T\vee A\right\}$ and characterized by the condition
	\[S\smallsetminus T\subseteq A\ \Leftrightarrow\ S\subseteq T\vee A.\]
	In particular, the co-pseudocomplement (usually called {\em supplement}) of $S$ is the sublocale $S^\#=L\smallsetminus S$ and we have that for any $S, T\in \Ss(L)$,
	\begin{equation}\label{fsuppl}
		S\vee T=L \ \Leftrightarrow\ S^\#\subseteq T\quad \text{ and }\quad S\cap T=\OS\ \Rightarrow\ S\subseteq T^\# \end{equation}
	(we refer to \cite{FPP} for more information on supplements in $\Ss(L)$). 
	
	For each $a\in L$, the sublocales
	\[
	\mathfrak{c}_L(a)=\ur  a=\{x\in L\mid x\ge a\}\quad \text{ and }\quad \mathfrak{o}_L(a)=\{a\to b\mid b\in L\}
	\]
	are the \emph{closed} and \emph{open} sublocales of $L$, respectively (that we shall denote simply by $\mathfrak{c}(a)$ and $\mathfrak{o}(a)$ when there is no danger of confusion). For each $a\in L$, $\mathfrak{c}(a)$ and $\mathfrak{o}(a)$ are
	complements of each other in $\mathsf{S}(L)$ and satisfy the identities
	\begin{equation}\label{closedsub}
		\tbigcap_{i\in J} \mathfrak{c}(a_i)=\close(\tbigvee_{i\in J} a_i),\quad \close(a)\vee\close(b)=\close(a\wedge b),
	\end{equation}
	\begin{equation}\label{opensub}\tbigvee_{i\in J}\open(a_i)=\open(\tbigvee_{i\in J} a_i) \quad\mbox{ and }\quad \open(a)\cap \open(b)=\open(a\wedge b).
	\end{equation}
	We denote by $\open(L)$ and $\close(L)$ the classes of all open and closed sublocales of $L$.
	
	Let $j_S^*$ be the left adjoint of the localic embedding $j_S\colon S\hookrightarrow L$, given by $j_S^*(a)=\tbigwedge \{s\in S\mid s\ge a\}$.
	The closed (resp. open) sublocales $\close_S(a)$ (resp. $\open_S(a)$) of $S$ ($a\in S$) are precisely the intersections $\close(a)\cap S$ (resp. $\open(a)\cap S)$
	and we have, for any $a\in L$,
	\begin{equation}\label{closedS}
		\close(a) \cap S = \close_S(j_S^*(a))\quad\mbox{and}\quad\open(a) \cap S = \open_S(j_S^*(a)).
	\end{equation}
	
	The {\em closure} $\overline S$ of a sublocale $S$ is  the smallest closed sublocale containing $S$, and the {\em interior} $S^\circ$ is
	the largest open sublocale contained in $S$. There is a particularly simple formula for the closure:
	\begin{equation}\label{closure}
		\overline S=\close(\tbigwedge S).
	\end{equation} Hence $\overline{\open(a)}=\close(a^*)$ and, consequently, $\close(a)^\circ=\open(a^*)$.
	More generally, for every sublocale $S$, $\overline{S^\#}=(S^\circ)^\#$ \cite[Section 4]{FPP}. A sublocale is {\em regular closed} (resp. {\em regular open}) if it is of the form $\close (a^*)$ (resp. $\open (a^*)$) for some $a\in L$.
	Note that, this allows us to look at the rather below relation in terms of sublocales:
	\begin{equation}\label{rb}
		b\prec a \iff \overline{\open (b)}\subseteq \open (a)\iff \close(a)\subseteq \close(b)^\circ.
	\end{equation}
	
	We say a sublocale is \emph{clopen} if it is both closed and open. A clopen sublocale is of the form $\close(a)$ (resp. $\open (a)$) for some \emph{complemented} $a\in L$ (that is, $a^*\vee a=0$).
	A sublocale $S$ of $L$ is {\em dense} if $\overline{S}=L$, equivalently, by \eqref{closure}, if $0_L\in S$. 
	
	\subsection{Image-Preimage adjunction of localic maps}
	
	Let $f\colon L\to M$ be a localic map. For any sublocale $S$ of $L$, its set theoretic image $f[S]$ is a sublocale of $M$. On the other hand, the set theoretic preimage $f^{-1}[T]$ of a sublocale $T$ of $M$ may not be a sublocale of $L$.
	But since $f$ is meet preserving, $f^{-1}[T]$ is closed under meets and thus there exists the
	largest sublocale of $L$ contained in $f^{-1}[T]$, usually denoted as $f_{-1}[T]$ (\cite[III.4]{PP12}). This is the {\em localic preimage} of $T$
	that provides the image/preimage Galois adjunction
	\[\xymatrixcolsep{5pc}\xymatrix{\Ss(L)\ar@/^/[r]^{f[-]}_\bot &\Ss(M)\ar@/^/[l]^{f_{-1}[-]}}\]
	between coframes $\Ss(L)$ and $\Ss(M)$ of sublocales of $L$ and sublocales of $M$, respectively.
	The right adjoint $f_{-1}[-]$ is a coframe homomorphism that preserves complements while $f[-]$ is a colocalic map (\cite[III.9]{PP12}). The adjunction between these two maps gives the following equivalence
	\begin{equation}
		f[S]\subseteq T\iff S\subseteq f_{-1}[T]
	\end{equation}
	for every $S\in \Ss(L)$ and $T\in \Ss(M)$.
	Localic preimages of closed (resp. open) sublocales are closed
	(resp. open). More specifically, denoting by $h$ the frame homomorphism left adjoint to $f$, we have
	\begin{equation}\label{preimage.close.open}
		f_{-1}[\close_M(a)]= \close_L(h(a))\quad \text{and}\quad f_{-1}[\open_M(a)]= \open_L(h(a))
	\end{equation}
	for any $a\in M$.
	
	In the particular case of the embedding $j\colon S \hookrightarrow L$ of a sublocale $S$ of $L$ the localic preimages are computed as the intersection; that is, $j_{-1}[T]=S\cap T$  for every $T\in \Ss(L)$.


	\subsection{Zero and cozero sublocales}
	An  $a\in L$ is a {\em cozero element}  if $a=\tbigvee_{n=1}^\infty a_n$ for some $a_n\pprec a$ ($n=1,2,\ldots$) where the {\em completely below relation} $\pprec$  is the interpolative modification of the rather below relation: $b\pprec a$ if and only if there exists a subset $\{a_q\mid q\in [0,1]\cap \QQ\}\subseteq L$ with $a_0=b$
	and $a_1=a$ such that $a_p\prec a_q$ whenever $p<q$. Recall also that a frame $L$ is said to be \emph{completely regular} if $a=\tbigvee\{b\in L\mid b\pprec a\}$ for every $a\in L$.\\
	Cozero elements form a $\sigma$-frame $\coz~L\subseteq L$, that is, a lattice  in which all {\em countable} 
	subsets have a join such that the distributivity law \eqref{frmdist}
	holds for any countable $S$. Furthermore, $\coz~L$ is a {\em normal} $\sigma$-frame, that is,  $a\vee b=1$ ($a,b\in \mathrm{Coz}\,L$) implies there exist $c$ and $d$ in $\coz~L$ such that $a\vee c=1=b\vee d$ and $c\wedge d=0$  \cite{BG}.
	
	The {\em zero sublocales} (resp. {\em cozero sublocales}) are the $\close(a)$ (resp. $\open(a)$) with $a\in \mathrm{Coz}\,L$.
	We denote by \[\mathsf{ZS}(L) \quad\mbox{ and }\quad \mathsf{CoZS}(L)\] the classes of zero and cozero sublocales, respectively. The former class is a sub-$\sigma$-coframe of $\Ss(L)$ while the latter is a sub-$\sigma$-frame. Since frame homorphisms preserve cozero elements, localic preimages of cozero (resp. zero) sublocales are cozero (resp. zero) sublocales.
	
	Furthermore, a sublocale $S$ is both a zero and a cozero sublocale iff it is clopen.

	\subsection{Completely separated sublocales}
	Two sublocales $S$ and $T$ of $L$ are \emph{completely separated in L} (\cite{GK,BW})\footnote{Complete separation is usually defined using real-valued functions; for simplicity we opt to use one of its characterizations as a defintion (see for instance \cite[2.7]{tesisA} for all the details).} if they are contained in disjoint zero sublocales of $L$; that is, if there exist $Z_1,Z_2\in \Sz(L)$ such that 
	\[S\subseteq Z_1,\quad T\subseteq Z_2 \quad \text{and}\quad Z_1\cap Z_2=\OS.\]
	
	\begin{remark}\label{cs.closure}
		\begin{enumerate}
			\item If $S\subseteq R$, $T\subseteq D$ and $R$ and $D$ are completely separated in $L$, then so are $S$ and $T$.
			\item $S$ and $T$ are completely separated if and only if $\overline{S}$ and $\overline{T}$ are completely separated. 
			\item Disjoint zero sublocales are completely separated. In particular, disjoint clopen sublocales are completely separated.
		\end{enumerate}
	\end{remark}

	\begin{proposition}\cite[5.4.2]{GPP-Notes}\label{cs.cb}
		Let $a$ and $b$ be elements of a locale $L$, $a\cb b$ if and only if $\open(a)$ and $\close(b)$ are completely separated in $L$.  
	\end{proposition}
	
	\begin{remark}\label{cs.normal}
		If $S$ and $T$ are completely separated in $L$ there exist $a,b\in \coz~L$ such that 
		\[S\subseteq \close(a),\quad T\subseteq \close(b) \quad \text{and}\quad \close(a)\cap \close(b)=\OS.\]
		Then, since $\coz~L$ is normal (which implies $\prec=\cb$) and $a\vee b=1$,  there exist $u,v\in \coz$ such that $u\wedge v=0$, $v\cb a $ and $u\cb b$. Hence, 
		\[S\subseteq \close(a)\subseteq \open(u),\quad T\subseteq \close(b)\subseteq \open(u)\quad\text{and}\quad \open(u)\cap \open(v)=\OS.\]
		Moreover, by Proposition~\ref{cs.cb}, $\open(u)$ and  $\close(b)$ are completely separated, and so are $\open(v)$ and  $\close(a)$. If we repeat this argument, then we can obtain another cozero sublocale $\open (y)$ such that $\open(u)$ and $\open (y)$ are completely separated and $S\subseteq \open (u)$ and $T\subseteq \open (y)$.
	\end{remark}
	
	\subsection{Special classes of localic maps}
	A localic map $f\colon L\to M$ is said to be 
	\begin{enumerate}
		\item {\em dense} if $f[L]$ is a dense sublocale of $M$. Equivalently, if $f(0_L)=0_M$. 
		\item {\em closed} if the image $f[S]$ is closed for each closed $S\in \Ss(L)$. Equivalently, if $f(a\vee h(b))=f(a)\vee b$ for every $a\in L$ and $b\in M$.
		\item {\em $\z$-closed} if the image $f[S]$ is closed for each zero sublocale $S\in \Ss(L)$. Equivalently, if $f(a\vee h(b))=f(a)\vee b$ for every $a\in \coz~L$ and $b\in M$.
		\item {\em $\rc$-closed} if the image $f[S]$ is closed for each regular closed sublocale $S\in \Ss(L)$. Equivalently, if $f(a\vee h(b))=f(a)\vee b$ for every $a\in L^*$ and $b\in M$.
		\item {\em $\z$-preserving} if the image $f[S]$ is a zero sublocale for each zero sublocale $S\in \Ss(L)$. Equivalently, if $f(a\vee h(b))=f(a)\vee b$ for every $a\in \coz~L$ and $b\in M$ and $f(a)\in \coz~M$ (\cite[Section~9]{parallel}). 
		\item {\em $\mathsf{reg}_\close$-preserving} if the image $f[S]$ is a regular closed for each regular closed sublocale $S\in \Ss(L)$. Equivalently, if $f(a\vee h(b))=f(a)\vee b$ for every $a\in L^*$ and $b\in M$ and $f(a)\in M^*$ (\cite[Section~9]{parallel}). 
		\item \emph{proper} if it is closed and preserves directed joins.
	\end{enumerate}
	Clearly, every closed map is $\z$-closed and $\rc$-closed. Moreover, $\z$-preserving (resp. $\rc$-preserving) maps are precisely the $\z$-closed (resp. $\rc$-preserving) maps such that $f(a)\in \coz~M$ for every $a\in \coz~L$ (resp. $f(a)\in M^*$ for every $a\in L^*$).
	\begin{remark}
		In general, 
		$\overline{f[S]}=\close_M(\tbigwedge f[S])=\close_M(f(\tbigwedge S))$, so $\overline{f[\close_L(a)]}=\close_M(f(a))$ for every $a\in L$. Hence, if $f[\close_L(a)]$ is closed for some $a\in L$, we have that $f[\close_L(a)]=\overline{f[\close_L(a)]}=\close_M(f(a))$.
	\end{remark}
	A localic map $f \colon L \to M$ is said to be {\em open} if the image $f[S]$ is open for each open $S\in \Ss(L)$. Other variants of openness have been studied in \cite{BanaPultr} and \cite{maps}, for instance nearly openness and $\co$-openness. Here we will use these notions and other variants. We say $f$ is 
	\begin{enumerate}
		\item {\em nearly open } if $h(a^*)=h(a)^*$ for every $a\in M$. Equivalently, $f_{-1}[\overline{\open(a)}]=\overline{f_{-1}[\open(a)]}$ for every $a\in L$.
		\item {\em $\co$-nearly open } if $h(a^*)=h(a)^*$ for every $a\in \coz~M$. Equivalently, $f_{-1}[\overline{\open(a)}]=\overline{f_{-1}[\open(a)]}$ for every $a\in \coz~L$.
		\item {\em $\ro$-nearly open } if $h(a^*)=h(a)^*$ for every $a\in M^*$. Equivalently, $f_{-1}[\overline{\open(a)}]=\overline{f_{-1}[\open(a)]}$ for every $a\in L$
		(this is condition $\mathbf{B}$ in \cite{BanaPultr}). 
		\item {\em $\co$-open} if $f[S]$ is open for each cozero sublocale $S\in \Ss(L)$.
		\item {\em $\ro$-open} if $f[S]$ is open for each regular open sublocale $S\in \Ss(L)$.
		\item {\em $\co$-preserving} if the image $f[S]$ is a cozero sublocale for each cozero sublocale $S\in \Ss(L)$. 
		\item $f$ is {\em $\mathsf{reg}_\open$-preserving} if the image $f[S]$ is regular open for each regular open sublocale $S\in \Ss(L)$. 
	\end{enumerate}
	
	Every open map is nearly open, and every nearly open map is $\co$-nearly open and $\ro$-nearly open. Furthermore, every open and every $\co$-preserving (resp. $\ro$-preserving) map is $\co$-open (resp. $\ro$-open).
	Note that $f$ is nearly open (resp. $\ro$-nearly open) if and only if  $f[\open_L(a)]\subseteq \overline{f[\open_L(a)]}^\circ$ for every $a\in L$ (resp. $a\in L^*$).

	

	\section{$\Gamma$-maps}
	
	Following the notation in \cite{open}, but adopting it to the localic setting, suppose $\gamma\colon \mathsf{Loc}\to \mathsf{CRegLoc}$ is a functor such that for each frame $L$ there is a localic map $\gamma_L\colon L\to \gamma L$ with the property that every localic map $f\colon L\to M$ lifts; that is, there is a localic map $f^\gamma\colon \gamma L\to \gamma M$ such that $\gamma_M f=f^\gamma\gamma_L$. If we consider the respective frame homorphisms (that is the left adjoints of the localic maps), we obtain the following diagram
	\begin{equation}\label{gamma1}
		\xymatrixcolsep{5pc}\xymatrixrowsep{5pc}\xymatrix{L \ar@<-1ex>[r]_{f}^\bot \ar@<1ex>[d]^{\gamma _L}_\dashv & M \ar@<-1ex>[l]_h\ar@<1ex>[d]^{\gamma _M}_\dashv\\ \gamma L \ar@<1ex>[u]^{\gamma _L^*} \ar@<-1ex>[r]_{f^\gamma}^\bot & \ar@<-1ex>[l]_{h^\gamma}\gamma M \ar@<1ex>[u]^{\gamma _M^*}} 
	\end{equation}
	where frame homomorphisms and localic maps commute respectively; that is  
	\begin{equation}\label{gamma2}
		h\cdot \gamma _M^*= \gamma _L^*\cdot h^\gamma\quad \text{and}\quad \gamma_M \cdot f = f^\gamma \cdot \gamma_L.
	\end{equation}
	We say a localic map $f\colon L\to M$ is a \emph{$\gamma$-map} if $h^\gamma\cdot \gamma_M=\gamma_L\cdot h$. Note that if $a\in M$ is such that $h^\gamma (\gamma_M(a))=\gamma_L(h(a))$, by taking the respective closed sublocales, this is equivalent to
	\begin{align*}
		f_{-1}^\gamma\left[\overline{\gamma _M[\close_M(a)]}^{\gamma M}\right] &=f_{-1}^\gamma[\close_{\gamma M} (\gamma_M(a))]=\close_{\gamma L}(h^\gamma (\gamma_M(a))) =\close_{\gamma L}(\gamma_L(h(a)))\\
		&= \overline{\gamma _L[\close_L(h(a))]}^{\gamma L}=\overline{\gamma _L\left[f_{-1}[\close_M(a)])\right]}^{\gamma L}.
	\end{align*}
	
	In general, we always have that $h^\gamma (\gamma_M(a))\leq \gamma_L(h(a))$ for any $a\in M$. Indeed, using the fact that $\gamma_L^*\dashv\gamma_L$ and \eqref{gamma2} we get
	\begin{align*}
		h^\gamma (\gamma_M(a))\leq \gamma_L(h(a)) &\iff \gamma_L^*(h^\gamma (\gamma_M(a)))\leq \gamma_L(h(a))\\
		&\iff  h(\gamma_M^*(\gamma_M(a)))\leq h(a)
	\end{align*}
	and $\gamma_M^*(\gamma_M(a))\leq a$ since $\gamma_M^*\dashv\gamma_M$.
	Equivalently, it is always true that 
	\[\overline{\gamma _L\left[f_{-1}[\close_M(a)])\right]}^{\gamma L}\subseteq f_{-1}^\gamma\left[\overline{\gamma _M[\close_M(a)]}^{\gamma M}\right].\]
	
	Instead of working only with elements and closed sublocales, we would like to extend the definition of $\gamma$-maps to other classes of sublocales. A \emph{selection function} on sublocales is a function which assigns to each locale $L$ a subset $\mathfrak{S}L$ of $\Ss(L)$. The standard selections we will use in this paper are taking closed sublocales, open sublocales, zero sublocales, cozero sublocales, regular closed sublocales and regular open sublocales. We denote them respectively as follows:
	\[\mathfrak{S}_\close \quad \mathfrak{S}_\open\quad\mathfrak{S}_\Sz\quad \mathfrak{S}_\Sc\quad \mathfrak{S}_{\mathsf{reg}\close} \quad \mathfrak{S}_{\mathsf{reg}\open}.\]
	\begin{definition}
		Given a selection function on sublocales $\mathfrak{S}$, we say a localic map $f\colon L\to M$ is a \emph{$\mathfrak{S}\gamma$-map} if 
		\begin{equation*}
			\overline{\gamma _L\left[f_{-1}[S]\right]}^{\gamma L}= f_{-1}^\gamma\left[\overline{\gamma _M[S]}^{\gamma M}\right].
		\end{equation*}
		holds for every $S\in \mathfrak{S}M$. 
	\end{definition}
	\begin{remark}\label{remark.g}
		Note that 
		\begin{align*}
			\overline{\gamma _L\left[f_{-1}[S]\right]}^{\gamma L}= \close_{\gamma L}\left(\tbigwedge \gamma_L[f_{-1}[S]]\right)= \close_{\gamma L}\left(\gamma_L(\tbigwedge f_{-1}[S])\right)
		\end{align*}
		and 
		\begin{align*}
			f_{-1}^\gamma\left[\overline{\gamma _M[S]}^{\gamma M}\right]=f_{-1}^\gamma\left[\close_{\gamma M}\left(\tbigwedge \gamma_M[S]\right)\right]=f_{-1}^\gamma\left[\close_{\gamma M}\left(\gamma_M(\tbigwedge S)\right)\right]=\close_{\gamma L}\left(h^\gamma(\gamma_M(\tbigwedge S))\right).
		\end{align*}
		Thus, $f$ is a $\mathfrak{S}\gamma$-map if and only if
		\[h^\gamma(\gamma_M (\tbigwedge S))= \gamma_L(\tbigwedge f_{-1}[S])\]
		For any sublocale $S$ of $M$ the inclusion 
		\begin{equation*}
			\overline{\gamma _L\left[f_{-1}[S]\right]}^{\gamma L}\subseteq f_{-1}^\gamma\left[\overline{\gamma _M[S]}^{\gamma M}\right].
		\end{equation*}
		always holds.  Indeed, due to the image/preimage adjunction, $f[f_{-1}[S]]\subseteq S$ thus
		\[h(\tbigwedge S)\leq h(\tbigwedge f[f_{-1}[S]])= h(f(\tbigwedge f_{-1}[S]))\leq \tbigwedge f_{-1}[S].\]
		Hence, 
		\begin{align*}
			h(\gamma_M^*(\gamma_M(\tbigwedge S)))\leq h(\tbigwedge S)\leq \tbigwedge f_{-1}[S] &\iff \gamma_L^*(h^\gamma(\gamma_M(\tbigwedge S)))\leq \tbigwedge f_{-1}[S]\\
			&\iff h^\gamma (\gamma_M(\tbigwedge S))\leq \gamma_L (\tbigwedge f_{-1}[S])\\
			&\iff \overline{\gamma _L\left[f_{-1}[S]\right]}^{\gamma L}\subseteq f_{-1}^\gamma\left[\overline{\gamma _M[S]}^{\gamma M}\right].
		\end{align*}
	\end{remark}
	
	\subsection{Two important $\gamma$-fications}
	The two $\gamma$-fications we will be working with are the Stone-\v{C}ech compactification $\beta$ and the Lindel\"{o}fication $\lambda$.  In general we do not require the locales to be completely regular. Both constructions still work with the exception that the maps $\beta_L$ and $\lambda_L$ may not be embeddings. When complete regularity is required it will be specified throughout the paper. We give a brief recap of the construction of these two $\gamma$-fications.
	
	Let us recall from \cite{BM} the Stone-\v{C}ech compatification of a frame $L$. An ideal $I$ of $L$ is regular if for every $a\in I$ there is a $b\in I$ such that $a\cb b$. For any $a\in L$, the ideal $\ddown a=\left\{x\in L\mid x\cb a\right\}$ is regular. The lattice $\mathfrak{R}(L)$ of all regular ideals of $L$ is a compact regular frame. Recall that a frame is \emph{compact} if $1=\tbigvee S$ implies $1=\tbigvee F$ for a finite $F\subseteq S$. The Stone-\v{C}ech compactification of $L$ is the dense localic map
	\begin{align*}
		\beta_L\colon L &\to \beta L=\mathfrak{R}(L)\\
		a&\mapsto \ddown a.
	\end{align*}
	The frame homorphism $\beta_L^*\colon \beta L\to L$ is given by taking joins of ideals. 
	If $L$ is completely regular $\beta_L$ is one-one localic map. 
	The functor $\beta\colon \mathsf{Loc} \to \mathsf{CRegKLoc}$ shows that compact regular locales are reflective in the category of locales. 
	For any localic map $f\colon L\to M$ its {\em Stone extension} is the localic map $ f^\beta \colon \beta L\to \beta M$ given by the functor $\beta$. The frame homomorphism $h^\beta\colon \beta M\to \beta L$, left adjoint of $f^\beta$, is easily computed by the formula:
	$h^\beta(I)= \downarrow h[I]$ 
	for any regular ideal $I$ of $M$. 
	The reflection gives the following diagram 
	\begin{equation}\label{diagram.b}
		\xymatrixcolsep{5pc}\xymatrixrowsep{5pc}\xymatrix{L \ar@<-1ex>[r]_{f}^\bot \ar@<1ex>[d]^{\beta _L}_\dashv & M \ar@<-1ex>[l]_h\ar@<1ex>[d]^{\beta _M}_\dashv\\ \beta L \ar@<1ex>[u]^{\beta _L^*} \ar@<-1ex>[r]_{f^\beta}^\bot & \ar@<-1ex>[l]_{h^\beta}\beta M \ar@<1ex>[u]^{\beta _M^*}} 
	\end{equation}
	where frame homomorphisms and localic maps commute respectively.

	Madden and Vermeer \cite{MV} show that through the functor $\lambda\colon \mathsf{Loc}\to \mathsf{CRLinLoc}$ the category of completely regular Lindel\"{o}f locales is reflective in the category of locales. A locale is \emph{Lindel\"{o}f} if $1=\tbigvee S$ implies $1=\tbigvee T$ for some countable $T\subseteq S$. 
	For any locale $L$, $\lambda L$ is the completely regular Lindel\"{0}f frame of all $\sigma$-ideals (that is, the ideals closed under countable joins) of $\coz~L$. For any $a\in L$ the set $\downarrow_{\co}a:=\left\{x\in \coz~L\mid x\leq a\right\}$ is a $\sigma$-ideal of $\coz~L$. The Lindel\"{o}fication and its respective frame homomorphism are given as follows:
	\[\begin{array}{rclcrcl}
		\lambda_L\colon L&\to &\lambda L&\quad & \lambda_L^*\colon \lambda L&\to &L\\
		a&\mapsto &\downarrow_{\co}a& & I&\mapsto &\tbigvee I
	\end{array}\]
	The map $\lambda_L$ is always dense, and it is one-one if $L$ is completely regular. For a localic map $f\colon L\to M$ the frame homorphism $h^\lambda$, left adjoint of the localic map $f^\lambda\colon \lambda L\to \lambda M$ given by $\lambda$, is described as follows:
	\begin{align*}
		h^\lambda\colon \lambda M&\to \lambda L\\
		I&\mapsto  \downarrow_{\co}h[I]
	\end{align*}
	where $\downarrow_{\co}h[I]=\left\{b\in \coz~L\mid b\leq h(x) \text{ for some } x\in I \right\}$.
	The reflection gives the following diagram 
	\begin{equation}\label{diagram.lambda}
		\xymatrixcolsep{5pc}\xymatrixrowsep{5pc}\xymatrix{L \ar@<-1ex>[r]_{f}^\bot \ar@<1ex>[d]^{\lambda _L}_\dashv & M \ar@<-1ex>[l]_h\ar@<1ex>[d]^{\lambda _M}_\dashv\\ \lambda L \ar@<1ex>[u]^{\lambda _L^*} \ar@<-1ex>[r]_{f^\lambda}^\bot & \ar@<-1ex>[l]_{h^\lambda}\lambda M \ar@<1ex>[u]^{\lambda _M^*}} 
	\end{equation}
	where frame homomorphisms and localic maps respectively, commute.
	
	\section{$\mathfrak{S}\beta$-maps}
	
	In this section we will study the $\mathfrak{S}\gamma$-maps when $\gamma$ is the Stone-\v{C}ech compactification. 
	The notion of $\mathfrak{S}\beta$-map (referred only as $\mathfrak{S}$-maps in \cite{maps}) is a generalization of the classical notions of $N$-, $WN$- and $WZ$-maps (\cite{Woods}). 
	$\mathfrak{S}_\close\beta$-maps are the right adjoints of $N$-homomorphisms of \cite{Remote} (or $\beta$-maps in \cite{open,closed}) and $\mathfrak{S}_\Sz\beta$-maps are the right adjoints of $WN$-homomorphisms in \cite{Remote}. For these two selections, the equivalence \ref{1}$\iff$ \ref{3} of the Theorem below gives as a corollary \cite[Lemma 5.2]{Remote} and \cite[6]{Remote}.
	
	\begin{theorem}\label{theorem.b}
		The following are equivalent for a localic map $f\colon L\to M$:
		\begin{enumerate}[label=\textup{(\roman*)}]
			\item\label{1} $f$ is an $\mathfrak{S}\beta$-map.
			\item\label{2} If an open sublocale $\open_L(a)$ is completely separated from $f_{-1}[S]$ for some $S\in \mathfrak{S}M$, then there is an open sublocale $\open_M(b)$ such that $\open_L(a)\subseteq f_{-1}[\open_M(b)]$ and $\open_M(b)$ is completely separated from $S$ in $M$. 
			\item\label{3}For every $a\in L$ such that $a\cb \tbigwedge f_{-1}[S]$ there exists $b\in M$ such that $a\leq h(b)$ and $b\cb\tbigwedge S$.
		\end{enumerate}
	\end{theorem}
	\begin{proof}
		\ref{1}$\iff$\ref{3}: By Remark~\ref{remark.g}, $f$ is $\mathfrak{S}\beta$-map if and only if 
		\[\beta_L(\tbigwedge f_{-1}[S])\leq h^\beta(\beta_M(\tbigwedge S)).\]
		Using the definitions of $\beta_L$ and $h^\beta$, one sees that this precisely means that
		\[\left\{y\in L\mid y\cb \tbigwedge f_{-1}[S]\right\}\subseteq \  \downarrow\left\{h(x)\mid x\in M,\ x\cb 0_S\right\}.\]
		Equivalenty, if for every $y\in L$ such that $y\cb \tbigwedge f_{-1}[S]$ there is an $x\in M$ such that $y\leq h(x)$ and $x\cb 0_S=\tbigwedge S$.\\ 
		\ref{2}$\iff$ \ref{3}: It follows from Proposition~\ref{cs.cb} and Remark~\ref{cs.cb}. Indeed, 
		\begin{align*}
			a\cb \tbigwedge f_{-1}[S]&\iff \open_L(a)\text{ is completely separated from } \close_L(\tbigwedge f_{-1}[S])=\overline{f_{-1}[S]}\\
			&\iff \open_L(a)\text{ is completely separated from } f_{-1}[S]\\ 
		\end{align*}
		and
		\begin{align*}
			b\cb \tbigwedge S&\iff \open_M(b)\text{ is completely separated from } \close_M(\tbigwedge S)=\overline{S}\\
			&\iff \open_M(b)\text{ is completely separated from } S.\\
		\end{align*}
		Furhtermore, from equations \eqref{opensub} and \eqref{preimage.close.open}, it is clear that $a\leq h(b)$ if and only if $\open_L(a)\subseteq \open_L(h(b))=f_{-1}[\open_M(b)]$.
	\end{proof}
	
	\begin{remark}\label{cs.beta}
		The fact that we are taking open sublocales in the characterization of $\mathfrak{S}\beta$-maps (Theorem~\ref{theorem.b}\ref{2}) is not relevant. In fact, one can take closed, zero, cozero, regular open or regular closed sublocales. 
		Indeed, by definition and Remark~\ref{cs.normal}, if $S$ and $T$ are completely separated there are two completely separated zero (or cozero) sublocales that contain $S$ and $T$. Furthermore, by Remarks~\ref{cs.normal} and \ref{cs.closure} one can also obtain two completely separated regular closed (or regular open) sublocales that contain $S$ and $T$. 
		Thus we get the following equivalent conditions to Theorem \ref{theorem.b}:
		\begin{itemize}
			\item[(iv)]  $Z$ is completely separated form $f_{-1}[S]$ for some $S\in \mathfrak{S}M$ and $Z\in \Sz(L)$, then there is $Z'\in \Sz(M)$ such that $Z\subseteq f_{-1}[Z']$ and $Z'$ is completely separated from $S$. 
			\item[(v)] $C$ is completely separated form $f_{-1}[S]$ for some $S\in \mathfrak{S}M$ and $C\in \Sc(L)$, then there is $C'\in \Sc(M)$ such that $C\subseteq f_{-1}[C']$ and $C'$ is completely separated from $S$. 
			\item[(vi)]  $\open_L(a^*)$ is completely separated form $f_{-1}[S]$ for some $S\in \mathfrak{S}M$, then there is a regular open sublocale $\open_M(b^*)$ such that $\open_L(a^*)\subseteq f_{-1}[\open_M(b^*)]$ and $\open_M(b^*)$ is completely separated from $S$. 
			\item[(vii)] $\close_L(a^*)$ is completely separated form $f_{-1}[S]$ for some $S\in \mathfrak{S}M$, then there is a regular closed sublocale $\close_M(b^*)$ such that $\close_L(a^*)\subseteq f_{-1}[\close_M(b^*)]$ and $\close_M(b^*)$ is completely separated from $S$. 
			\item[(viii)] $\close_L(a)$ is completely separated form $f_{-1}[S]$ for some $S\in \mathfrak{S}M$, then there is a closed sublocale $\close_M(b)$ such that $\close_L(a)\subseteq f_{-1}[\close_M(b)]$ and $\close_M(b)$ is completely separated from $S$. 
		\end{itemize}
	\end{remark}
	
	\begin{remark}
		\begin{enumerate}
			\item For any two selections $\mathfrak{S}_1$ and $\mathfrak{S}_2$, if $\mathfrak{S}_1\subseteq \mathfrak{S}_2$ (that is, if $\mathfrak{S}_1L\subseteq \mathfrak{S}_2L$ for every frame $L$) then any $\mathfrak{S}_2\beta$-map is an $\mathfrak{S}_1\beta$-map. 
			\item A localic map $f\colon L\to M$ is nearly open precisely when $f_{-1}[\overline{\open(a)}]=\overline{f_{-1}[\open(a)]}$ for every $a\in L$, and two sublocales are completely separated if and only if their closures are completely separated. Thus, if a localic map $f$ is nearly open, then it is an $\mathfrak{S}_{\rc}\beta$-map if and only if it is $\mathfrak{S}_\open\beta$.
		\end{enumerate}
	\end{remark}

	\subsection{More characterizations of variants of normality}
	Let us provide characterizations of variants of normality using $\mathfrak{S}\beta$-maps. 
	A frame $L$ is \emph{$\mathfrak{S}$-normally separated} when every $S\in\mathfrak{S}L$ is completely separated from every closed sublocale disjoint from it. For instance, $\mathfrak{S}_\close$-normally separated locales are precisely the normal locales, and $\mathfrak{S}_\Sz$-normally separated locales correspond to $\delta$-normally separated locales (see for instance \cite{Remote,maps}). 
	A further example is that of almost normal frames. Recall from \cite{perfect,Separation} that a locale $L$ is \emph{almost normal} if for every $a,b^*\in L$ such that $a\wedge b^*=0$ there are $u,v\in L$ such that $a\vee u=1=v\wedge b^*$ and $u\wedge v =1$. Almost normal frames are precisely the $\mathfrak{S}_{\mathsf{reg}\close}$-locales:
	\begin{proposition}
		The following are equivalent for a locale $L$: 
		\begin{enumerate}[label=\textup{(\roman*)}]
			\item\label{an1} $L$ is almost normal.
			\item\label{an2} The rather below relation $\prec$ interpolates (consequently, $\prec=\cb$).
			\item\label{an3} If $\close(a)\cap \close(b^*)=\OS$ then $\close(a)$ and $\close(b^*)$ are completely separated.
			\item\label{an4} If $\close(a)\subseteq\close(b)^\circ$ for $a,b\in L$, then there $x\in L$ such that \[\close(a)\subseteq \open(x)\subseteq\overline{\open(x)} \subseteq \close(b)^\circ.\]
			\item\label{an5} If $\overline{\open(b)}\subseteq\open(a)$ for $a,b\in L$, then there $x\in L$ such that 
			\[\overline{\open(b)}\subseteq \open(x)\subseteq\overline{\open(x)}\subseteq \open(a).\]
		\end{enumerate}
	\end{proposition}
	\begin{proof}
		\ref{an1}$\iff$\ref{an2} is proved in \cite[VII.1.3.2]{Separation}.\\
		\ref{an2}$\iff$\ref{an3}: It suffices to note two facts. Firstly, recall that $\close(a)\cap \close(b^*)=\OS$ if and only if $a\vee b^*=1$, that is $b\prec a$. Secondly, by \ref{an2}, $b\cb a$ means $\open (b)$ and $\close(a)$ are completely separated (recall Proposition~\ref{cs.cb}). Equivalently, by Remark~\ref{cs.closure}, $\overline{\open(b)}=\close(b^*)$ and $\close(a)$ are completely separated. \\
		\ref{an3}$\implies$\ref{an4}: Let $\close(a)\subseteq \close(b)^\circ$ for some $a,b\in L$, then $\close(a)$ and $\close(b^*)$ are disjoint; by assumption, they are completely separated. Thus, by Remark~\ref{cs.normal}, there are $x,y\in \coz~L$ such that $\close(a)\subseteq \open (y)$, $\close(b^*)\subseteq \open(x)$ and $\open (x)\cap \open (y)=\OS$. Hence, \[\close_L(a)\subseteq \open _L(y)\subseteq \overline{\open (y)}\subseteq \close_L(x)\subseteq \close(b)^\circ.\] 
		\ref{an4}$\implies$\ref{an5}: Let $\overline{\open(b)}\subseteq\open(a)$ for $a,b\in L$, then $\close(a)\subseteq\close(b)^\circ$. By assumption, there is $y\in L$, such that $\close(a)\subseteq\open (y)\subseteq \overline{\open(y)}\subseteq\close(b)^\circ$. Taking complements we obtain 
		\[\overline{\open(b)}\subseteq \open (y^*)\subseteq \close(y)\subseteq\open(a)\]
		and $\open (y^*)\subseteq\overline{\open (y^*)}\subseteq \close(y)$.\\
		\ref{an5}$\iff$\ref{an2}: This is very easily seen from equation \eqref{rb}: for any $a,b\in L$ 
		\[b\prec a\iff \overline{\open(b)}\subseteq\open(a)\]
		and $\prec$ interpolates if for any $b\prec a$ there is $x\in L$ such that \[b\prec x\prec a \iff \overline{\open(b)}\subseteq\open(x)\subseteq \overline{\open(x)}\subseteq \open (a).\]
	\end{proof}

	The following Theorem was proved in \cite[Theorem 5.3]{maps} and it is a generalization of the results in \cite{Woods}, for sake of completeness we write it below:
	
	\begin{theorem}\label{theorem.bmap}
		Let $\mathfrak{S}$ be a selection on sublocales, and let $M$ be a locale such that every $T\in\mathfrak{S}M$ is complemented. The following are equivalent:
		\begin{enumerate}[label=\textup{(\roman*)}]
			\item $M$ is $\mathfrak{S}$-normally separated. 
			\item Every $\z$-closed localic map $f\colon L\to M$ is an $\mathfrak{S}\beta$-map. 
			\item Every closed localic map $f\colon L\to M$ is an $\mathfrak{S}\beta$-map. 
			\item Every proper localic map $f\colon L\to M$ is an $\mathfrak{S}\beta$-map. 
		\end{enumerate}
	\end{theorem}
	
	One applies Theorem~\ref{theorem.bmap} to the selections $\mathfrak{S}_\close$, $\mathfrak{S}_\Sz$ and $\mathfrak{S}_{\rc}$ and obtains a characterization for normal, $\delta$-normally separated (see Corollaries 5.4 and 5.5 of \cite{maps} or \cite{Remote}) and almost normal frames.
	
	Recall from \cite{Remote} that a locale $L$ is \emph{weakly $\delta$-normally} separated if whenever $\close(a)\cap \close(b)=\OS$ with $a\in \coz~L$ and $b\in L^*$, then $\close(a)$ and $\close(b)$ are completely separated in $L$. 
	A locale $L$ is \emph{mildly normal} if for every regular elements $a$ and $b$ in $L$ such that $a\vee b = 1$, there exist $c,d\in L$ such that $c\wedge d = 0$ and $a\vee c = b\vee d = 1$ (\cite{Mthesis}); equivalently, if disjoint regular closed sublocales are completely separated. 
	\begin{proposition}
		Let $M$ be a locale. 
		\begin{enumerate}[label=\textup{(\arabic*)}]
			\item\label{wdn} $M$ is weakly $\delta$-normally separated if and only if every $\rc$-preserving map $f\colon L\to M$ is an $\mathfrak{S}_\Sz$-map.
			\item\label{an} $M$ is almost normal if and only if every $\rc$-preserving map $f\colon L\to M$ is an $\mathfrak{S}_\close\beta$-map.
			\item\label{mn} $M$ is mildly normal if and only if every $\rc$-preserving map $f\colon L\to M$ is an $\mathfrak{S}_{\mathsf{reg}\close}$-map.
		\end{enumerate}
	\end{proposition}
	\begin{proof}
		\ref{wdn}: Let $M$ be a weakly $\delta$-normally separated frame and let $f\colon L\to M$ be a $\rc$-preserving map. Let $a\in L$ and $b\in \coz M$ such that $\overline{\open_L(a)}$ is completely separated from $f_{-1}[\close_M(b)]$. In particular we have \[\overline{\open_L(a)}\subseteq f_{-1}[\close_M(b)]^\#=f_{-1}[\open_M(b)].\]
		Thus, by the image/preimage adjunction $f[\overline{\open_L(a)}]\subseteq \open_M(b)$. That is, $f[\overline{\open_L(a)}]\cap \close_M(b)=\OS$. Since $f$ is $\rc$-preserving, $f[\overline{\open_L(a)}]$ is regular closed, and $M$ is weakly $\delta$-normally separated so $f[\overline{\open_L(a)}]$ and $\close_M(b)$ are completely separated in $M$, which proves that $f$ is a $\mathfrak{S}_\Sz$-map.\\
		Conversely, assume that every $\rc$-preserving map $f\colon L\to M$ is an $\mathfrak{S}_{\Sz}$-map. Let $a\in M$ and $b\in \coz~M$ such that $\overline{\open_M(a)}\cap \close_M(b)=\OS$. Let us right $S:=\overline{\open_M(a)}=\close_M(a^*)$.
		Consider the embedding $j\colon \overline{\open_M(a)}\hookrightarrow M$; it is $\rc$-closed. Indeed take a regular element $x^{*_S}$ of $S$. Using properties of the Heyting operator (\cite[Proposition 3.1.1]{PP12}), we know 
		\[x^{*_S}=x\rightarrow a^*=x\rightarrow (a\rightarrow 0)=(x\wedge a)\rightarrow 0=(x\wedge a)^*.\]
		Thus $x^{*_S}$ is a regular element of $M$. Then, since closed sublocales of closed sublocales are closed,
		$j[\close_S(x^{*_S})]=\close_S(x^{*_S})=\close_M(x^{*_S})$
		is a regular closed sublocale of $M$. Thus, $j$ is $\rc$-preserving and hence a $\mathfrak{S}_\Sz$-map.
		Note that $\overline{\open_M(a)}$ and $j{-1}[\close_M(b)]=\OS$ are trivially completely separated in $S$. Then there exists a zero sublocale $Z\in \Sz(M)$ such that $\overline{\open_M(a)}\subseteq Z$ and $Z$ is completely separated from $\close_M(b)$. In particular, $\overline{\open_M(a)}$ and $\close_M(b)$ are completely separated in $M$. \\
		For \ref{an} and \ref{mn} one mimics the proof of \ref{wdn} considering $b\in M$ or $b^*\in M$, and substituting $\mathfrak{S}_\Sz$ for $\mathfrak{S}_\close$ and $\mathfrak{S}_{\mathsf{reg}\close}$, respectively. 
	\end{proof}
	
	
	The following result is an improvement of Proposition 6.1 in \cite{Remote} where the authors impose two extra conditions, namely assertiveness and $coz$-closedness\footnote{A word of caution to the reader: $coz$-closed maps of \cite{Remote} are not the same as $\z$-closed maps, as one would expect. More will be commented in Remark~\ref{coz.closed}.}.
	As you can see this result is very similar to the classical one in \cite[Lemma 1.4]{Woods}:
	
	\begin{proposition}
		Let $f\colon L\to M$ be an onto $\mathfrak{S}_\Sz$-localic map.
		\begin{enumerate}[label=\textup{(\arabic*)}]
			\item\label{dw1} If $L$ is $\delta$-normally separated then so is $M$.
			\item\label{dw2} If $L$ is weakly $\delta$-normally separated then so is $M$.
		\end{enumerate}
	\end{proposition}
	\begin{proof}
		\ref{dw1}: Take $Z\in \Sz(M)$ and $a\in M$ such that $Z\cap \close_L(a)=\OS$. Then $f_{-1}[Z]\cap f_{-1}[\close_M(a)]=\OS$. Since $L$ is $\delta$-normally separated, there is $D\in \Sz(L)$ such that $f_{-1}[Z]\cap D =\OS$ and $f_{-1}[\close_L(a)]\subseteq D$. Since $f$ is a $\mathfrak{S}_\Sz$-map there is $Z'\in \Sz(M)$ such that $Z$ and $Z'$ are completely separated in $M$ and $D \subseteq f_{-1}[Z']$. we have 
		\[f_{-1}[\close_M(a)]\subseteq D\subseteq f_{-1}[Z']\]
		then $f[f_{-1}[\close_M(a)]]\subseteq Z'$. Thus, $Z$ and $f[f_{-1}[\close_M(a)]]$ are completely separated in $M$. In fact, by Remark~\ref{cs.closure}, $Z$ and $\overline{f[f_{-1}[\close_M(a)]]}$ are completely separated in $M$. Notice that 
		\[\overline{f[f_{-1}[\close_M(a)]]}= \overline{f[\close_L(h(a))]]}=\close_M(f(h(a)))=\close_M(a)\]
		where the last equality holds because $f$ is onto. Hence, $Z$ and $\close_M(a)$ are completely separated in $M$, as required. \\
		The proof of \ref{dw2} follows similarly as in \ref{dw1}, but one has to be careful with regular elements.
		Let $Z\in \Sz(M)$ and $a\in M$ such that $Z\cap \close_L(a^*)=\OS$, then $f_{-1}[Z]\cap f_{-1}[\close_M(a^*)]=\OS$. Note that 
		\[\close_L(h(a)^*)=\overline{f_{-1}[\open_M(a)]}\subseteq f_{-1}[\overline{\open_M(a)}]=f_{-1}[\close_M(a^*)].\]
		Hence, $f_{-1}[Z]\cap \close_L(h(a)^*)=\OS$.
		Since $L$ is  weakly $\delta$-normally separated, there is $D\in \Sz(L)$ such that $f_{-1}[Z]\cap D =\OS$ and $\close_L(h(a)^*)\subseteq D$. Since $f$ is a $\mathfrak{S}_\Sz\beta$-map there is $Z'\in \Sz(M)$ such that $Z$ and $Z'$ are completely separated in $M$ and $D \subseteq f_{-1}[Z']$. We have that
		\[\close_L(h(a)^*)\subseteq D\subseteq f_{-1}[Z'],\]
		so using the image/preimage adjunction and the fact that $Z'$ is closed, we get $\overline{f[\close_L(h(a)^*)]}\subseteq Z'$. Consequently, $Z$ and $\overline{f[\close_L(h(a)^*)]}$ are completely separated in $M$. Notice that $\overline{f[\close_L(h(a)^*)]}=\close_M(f(h(a)^*))$ and \[f(h(a)^*)=f(h(a)\to 0)=a\to f(0)=a\to 0=a^*\]
		where the second and third equality hold because $f$ is an onto localic map. Thus, $Z$ and $\close_M(a^*)$ are completely separated in $M$, as required.
	\end{proof}
	
	It is well-known that the image of a normal locale under a closed map is normal (\cite[VIII.5.2.2]{Separation}). In fact, mimicking the proof of the previous proposition, we also obtain the following result:
	
	\begin{proposition}
		Let $f\colon L\to M$ be an onto $\mathfrak{S}_\close\beta$-localic map.
		\begin{enumerate}
			\item If $L$ is normal then so is $M$.
			\item If $L$ is almost normal then so is $M$.
		\end{enumerate}
	\end{proposition}
	
	\subsection{Characterizing variants of disconnectedness}
	
	
	We have generalized the notion of $\beta$-maps using selection of sublocales; it is only natural to dualize the notions and results of the previous subsection to selections of open sublocales.\\
	
	We say a frame $L$ is \emph{$\mathfrak{S}$-disconnected} if $S\in\mathfrak{S}L$ is completely separated from every open sublocale disjoint from it. 
	
	\begin{remark}\label{disc}
		A locale $L$ is $\mathfrak{S}$-disconnected if and only if $\overline{S}$ is clopen for every $S\in \mathfrak{S}L$. Indeed, in general we have $S\cap \overline{S}^\#=\OS$ for any $S\in \mathfrak{S}L$; by $\mathfrak{S}$-disconnectedness, $S$ and $\overline{S}^\#$ are completely separated. In particular, there exist $Z\in \Sz(L)$ and $C\in \Sc(L)$ such that $S\subseteq Z\subseteq C\subseteq \overline{S}$. Hence,
		\[\overline{S}^\circ\subseteq\overline{S}\subseteq Z\subseteq C\subseteq \overline{S}^\circ.\]
		Conversely, suppose $S\cap \open(a)=\OS$ for some $S\in\mathfrak{S}L$ and $a\in L$, then $\overline{S}\cap \open(a)=\OS$. Clearly $S\subseteq \overline{S}$ and $\open(a)\subseteq \overline{S}^\#$. By assumption, $\overline{S}$ and $\overline{S}^\#$ are disjoint clopens, and hence disjoint zero sublocales; thus, they completely separate $S$ and $\open(a)$.	
	\end{remark}
	We obtain the dual result of Theorem~\ref{theorem.bmap}:
	\begin{theorem}\label{ch.op}
		Let $\mathfrak{S}$ be a selection on sublocales, and let $M$ be a locale such that every $T\in\mathfrak{S}M$ is complemented. The following are equivalent:
		\begin{enumerate}[label=\textup{(\roman*)}]
			\item\label{e.1} $M$ is $\mathfrak{S}$-disconnected. 
			\item\label{e.2} Every $\co$-open localic map $f\colon L\to M$ is an $\mathfrak{S}\beta$-map. 
			\item\label{e.3} Every open localic map $f\colon L\to M$ is an $\mathfrak{S}\beta$-map. 
		\end{enumerate}
	\end{theorem}
	\begin{proof}
		\ref{e.1}$\implies$\ref{e.2}: Assume $M$ is $\mathfrak{S}$-disconnected. Let $f\colon L\to M$ be a $\co$-open localic map. Let $C\in \Sc(L)$ and $S\in\mathfrak{S}M$ such that $C$ and $f_{-1}[S]$ are completely separated in $L$. In particular, $C\subseteq f_{-1}[S]^\#=f_{-1}[S^\#]$. By the image/ preimage adjunction, we have that $f[C]\subseteq S^\#$. Since every sublocale in $\mathfrak{S}M$ is complemented, we get $f[C]\cap S=\OS$. By the $\co$-openness of $f$ and the $\mathfrak{S}$-disconnectedness of $M$, we know $f[C]$ and $S$ are completely separated in $M$. In particular, by Remark~\ref{cs.normal},  there is a cozero sublocale $D\in\mathsf{CozS}(M)$ such that $D$ is completely separated from $S$ and $f[C]\subseteq D$ (equivalently, $C\subseteq f_{-1}[D]$).\\	
		Clearly \ref{e.2}$\implies$ \ref{e.3}, since very open localic map is $\co$-open.\\
		\ref{e.3}$\implies$\ref{e.1}: Let $S\in \mathfrak{S}M$ and $a\in M$ such that $S\cap \open_M(a)=\OS$. Take the open localic embedding $j\colon\open_M(a)\hookrightarrow L$; by assumption, it is an $\mathfrak{S}\beta$-map. Trivially, $\open_M(a)$  and $j_{-1}[S]$ are completely separated in $\open_M(a)$ (since $j_{-1}[S]=S\cap \open_M(a)=\OS$). Hence, there is $b\in M$ such that $\open_M(b)$ and $S$ are completely separated in $M$, and 
		$\open_M(a)\subseteq j_{-1}[\open_M(b)]=\open_M(a)\cap \open_M(b)$. Thus, $\open_M(a)\subseteq \open_M(b)$, meaning $\open_M(a)$ and $S$ are also completely separated, showing that $M$ is $\mathfrak{S}$-disconnected. 
	\end{proof}
	
	Recall that a frame $L$ is \emph{extremally disconnected} (resp. basically disconnected) if $a^*\vee a^{**}=1$ for every $a\in L$ (resp. $a\in\coz~L$). 
	Using the characterizations in \cite{parallel} and \cite[Proposition 8.4.3]{BW} one can easily see that $\mathfrak{S}_\open$-disconnectedness is just extremal disconnectedness, and a locale is $\mathfrak{S}_\Sc$-disconnected precisely if it is basically disconnected. Thus, applying Theorem~\ref{ch.op} to these selections we obtain the following two corollaries: 
	
	\begin{corollary}
		The following are equivalent for a locale $M$:
		\begin{enumerate}[label=\textup{(\roman*)}]
			\item $M$ is extremally disconnected.
			\item Every $\co$-open localic map $f\colon L\to M$ is an $\mathfrak{S}_\open\beta$-map.
			\item Every open localic map $f\colon L\to M$ is an $\mathfrak{S}_\open\beta$-map.
		\end{enumerate}
	\end{corollary}
	\begin{corollary}
		The following are equivalent for a locale $M$:
		\begin{enumerate}[label=\textup{(\roman*)}]
			\item $M$ is basically disconnected.
			\item Every $\co$-open localic map $f\colon L\to M$ is an $\mathfrak{S}_\Sc\beta$-map.
			\item Every open localic map $f\colon L\to M$ is an $\mathfrak{S}_\Sc\beta$-map.
		\end{enumerate}
	\end{corollary}
	
	A locale $L$ is said to be an \emph{$F$-frame} if every pair of disjoint cozero sublocales are completely separated in $L$ (\cite[Proposition 8.4.10]{BW}).
	\begin{proposition}
		Let $M$ be a locale.
		\begin{enumerate}[label=\textup{(\arabic*)}]
			\item\label{f} $M$ is an $F$-frame if and only if every $\co$-preserving map is an $\mathfrak{S}_\Sc\beta$-map. 
			\item\label{bd} $M$ is basically disconnected if and only if every $\co$-preserving map is an $\mathfrak{S}_\open\beta$-map. 
		\end{enumerate}
	\end{proposition}
	\begin{proof}
		Let $M$ be an $F$-frame (resp. basically disconnected frame) and $f\colon L\to M$ be a $\co$-preserving localic map. Let $C\in \Sc(L)$ and $D\in \Sc(M)$ (resp. $D$ and open sublocale of $M$) such that $C$ and $f_{-1}[D]$ are completely separated. In particular, $C\cap f_{-1}[D]=\OS$. Hence, $C\subseteq f_{-1}[D^\#]$. By the image/preimage adjunction $f[C]\subseteq D^\#$, equivalently $f[C] \cap D=\OS$, and $f[C]$ is a cozero sublocale ($f$ is $\co$-preserving).
		Since $M$ is an $F$-frame (resp. basically disconnected), $f[C]$ and $D$ are completely separated and $C\subseteq f_{-1}[f[C]]$.\\
		Conversely, suppose every $\co$-preserving map is an $\mathfrak{S}_\Sc\beta$-map (resp. $\mathfrak{S}_\open\beta$-map). Let $C_1,C_2\in \Sc(L)$ (resp. $C_1\in \Sc(L)$ and $C_2$ an open sublocale) such that $C_1\cap C_2=\OS$. Consider the localic embedding $j\colon C_1\hookrightarrow L$. It is a $\co$-preserving map. Indeed, every $C\in \Sc(C_1)$, since $C_1$ is a cozero sublocale, is a cozero sublocale of $L$ (recall \cite[Proposition ]{BW}); that is, $j[C]=C\in \Sc(L)$. Trivially $C_1$ and $j_{-1}[C_2]$ are completely separated in $C_1$, since $j_{-1}[C_2]=C_1\cap C_2=\OS$. By assumption, there is $ D\in \Sc(L)$ such that $C_1\subseteq D$ and $D$ is completely separated from $C_2$. In particular, $C_1$ and $C_2$ are completely separated, as required. 
		
	\end{proof}
	
	A localic map $f\colon L\to M$ is \emph{$\cb$-preserving} if $a\cb b$ implies $f(a)\cb f(b)$ for every $a,b\in L$. The frame homomorphisms corresponding to these maps were introduced in \cite{Remote} and are called assertive. It seems natural (recall Theorem~\ref{theorem.b}~\ref{3}) to ask if $\mathfrak{S}\beta$-maps are related to $\cb$-presserving maps. In \cite[Examples 3.13]{closed} the authors give examples showing that the notions of $\mathfrak{S}_\close\beta$-maps and $\cb$-preserving are not related. Nevertheless, $\cb$-preserving are precisely the $\mathfrak{S}_\open\beta$-maps.
	\begin{theorem}\label{ass}
		Let $f\colon L\to M$ be a localic map, $f$ is an $\mathfrak{S}_\open\beta$-map if and only if it is $\cb$-preserving.
	\end{theorem}
	\begin{proof}
		Let $a,b\in L$ such that $a\cb b$; in other words, $\open_L(a)$ is completely separated from $\close_L(b)$ (Proposition~\ref{cs.cb}). We will show $\open_M(f(a))$ and $\close_M(f(a))$ are completely separated in $M$. Note that:
		\[\close_L(a)\subseteq f_{-1}[f[\close_L(a)]]\subseteq f_{-1}[\overline{f[\close_L(a)]}]=f_{-1}[\close_M(f(a))]\]
		By taking complements, we obtain $f_{-1}[\open_M(f(a))]\subseteq \open_L (a)$.
		Thus, $f_{-1}[\open_M(f(a))]$ and $\close_L(b)$ are completely separated in $L$. Since $f$ is an $\mathfrak{S}_\open\beta$-map, there is $x\in M$ such that $\close_M(x)$ is completely separated from $\open_M(f(a))$ in $M$, and $\close_L(b)\subseteq f_{-1}[\close_M(x)]$. Furthermore, 
		\[\close_L(b)\subseteq f_{-1}[\close_M(x)]\iff f[\close_L(b)]\subseteq \close_M(x)\iff \close_M(f(b))=\overline{ f[\close_L(b)]}\subseteq \close_M(x).\]
		Hence, $\close_M(f(b))$ and $\open_M(f(a))$ are completely separated in $M$, as required. \\
		Conversely, assume $f$ is $\cb$-preserving and $b\in L$ and $a\in M$ such that $\close_L(b)$ and $f_{-1}[\open_M(a)]$ are completely separated. Equivalently, $h(a)\cb b$, where $h$ is the left adjoint (frame homomorphism) of $f$. By assumption, $ f(h(a))\cb f(b)$. Since $a\leq f(h(a))$, we have $a\cb f(b)$ meaning $\open_L(a)$ is completely separated from $\close_M(f(b))$. Furtheremore, 
		\[\close_L(b)\subseteq \close_L(h(f(b)))= f_{-1}[\close_M(f(b))]\]
		again for the inclusion above we are using the fact that $f$ and $h$ are adjoints.
	\end{proof}
	
	Using the definition of $\mathfrak{S}\beta$-map and Theorem~\ref{theorem.b}, we put together all the characterizing conditions of $\cb$-preserving maps. Recall the notation in diagram~\eqref{diagram.b} and the fact that $\tbigwedge \open (a)=a^*$,
	\begin{corollary}
		Let $f\colon L\to M$ be a localic map, the following are equivalent:
		\begin{enumerate}[label=\textup{(\roman*)}]
			\item $f$ is $\cb$-preserving.
			\item $\overline{\beta _L\left[f_{-1}[\open_M(a)]\right]}^{\beta L}= f_{-1}^\beta\left[\overline{\beta _M[\open_M(a)]}^{\beta M}\right]$ for every $a\in M$.
			\item If an open sublocale $\open_L(a)$ is completely separated from $f_{-1}[\open_M(x)]$ for some $x\in M$, then there is an open sublocale $\open_M(b)$ such that $\open_L(a)\subseteq f_{-1}[\open_M(b)]$ and $\open_M(b)$ is completely separated from $\open_M(x)$ in $M$. 
			\item $h^\beta(\beta_M(x^*))=\beta_L(h(x)^*)$ for every $x\in M$.
			\item For every $a\in L$ such that $a\cb h(x)^*$ there exists $b\in M$ such that $a\leq h(b)$ and $b\cb x^*$.
		\end{enumerate}
	\end{corollary}
	
	\begin{remark}\label{clop.ass}
		Clearly, if a map is $\cb$-preserving the image of complemented elements is complemented. Thus, if $f\colon L\to M$ is an $\mathfrak{S}_\open\beta$-map then $\overline{f[S]}$ is clopen for every clopen sublocale $S$ of $L$. 
	\end{remark}
	\begin{proposition}\label{no}
		Let $L$ be a an $\mathfrak{S}$-disconnected locale. If $f\colon L\to M$ is an $\mathfrak{S}_\open\beta$-map then $\overline{f[S]}$ is clopen for every $S\in \mathfrak{S}L$.
	\end{proposition}
	\begin{proof}
		Let $S\in \mathfrak{S}L$, by Remark~\ref{disc}, $\overline{S}$ is clopen. Hence, using Theroem~\ref{ass} (in particular, Remark~\ref{clop.ass}), we get $\overline{f[\overline{S}]}=\overline{f[S]}$ is clopen.
	\end{proof}
	
	From Proposition~\ref{no} one gets that if $L$ is extremally diconnected, then every localic $\mathfrak{S}_\open\beta$-map $f\colon L\to M$ is nearly open.\\ 
	
	
	We know that the image of an extremally disconnected frame under an open localic map is extremally disconnected. In fact, we obtain the following result:
	
	\begin{proposition}\label{ed.onto}
		Let $f\colon L\to M$ be a dense $\mathfrak{S}_\open\beta$-localic map.
		\begin{enumerate}
			\item If $L$ extremally disconnected then so is $M$.
			\item If $L$ basically disconnected then so is $M$.
		\end{enumerate}
	\end{proposition}
	\begin{proof}
		Let $f\colon L\to M$ be an onto $\mathfrak{S}_\open\beta$-localic map, and let $L$ be an extremally disconnected (resp. basically disconnected) locale.
		Let $a,b\in M$ (resp. $a\in M$ and $b\in \coz~M$) such that $\open_M(a)\cap \open_M(b)=\OS$, then $ f_{-1}[\open_M(a)]\cap f_{-1}[\open_M(b)]=\OS$. Since $L$ is extremally disconnected (resp. basically disconnected), $f_{-1}[\open_M(a)]$ and  $f_{-1}[\open_M(b)]$ are completely separated in $L$. By Remark~\ref{cs.closure}, $\overline{f_{-1}[\open_M(a)]}$ and  $f_{-1}[\open_M(b)]$ are completely separated. We know $f$ is an $\mathfrak{S}_\open\beta$-localic map, so there is $x\in M$ such that $\close_M(x)$ is completely separated from $\open_M(b)$ and 
		\begin{equation*}\label{ppp}
			\overline{f_{-1}[\open_M(a)]}\subseteq f_{-1}[\close_M(x)].
		\end{equation*}
		Thus, by the image/preimage adjunction and the fact that $\close_M(x)$ is closed, we have that  \[\overline{f[\overline{f_{-1}[\open_M(a)]}]}~\subseteq~\close_M(x). \]Note that 
		\begin{equation*}
			\overline{f[\overline{f_{-1}[\open_M(a)]}]}= \overline{f[\overline{\open_L(h(a))}]}= \overline{f[\close_L(h(a)^*)]}=\close_M(f(h(a)^*))=\close_M(a^*)=\overline{\open_M(a)}
		\end{equation*}
		where the fourth equality holds because $f$ is dense which implies that $f(0)=0$ and $f(h(a)^*))=f(h(a)\rightarrow 0)=a\rightarrow f(0)= a\rightarrow 0=a^*$. 
		Hence, since $\open_M(a)\subseteq \overline{\open_M(a)}\subseteq \close_M(x)$, the sublocales $\open_M(a)$ and $\open_M(b)$ are completely separated in $M$ showing that $M$ is extremally disconnected (resp. basically disconnected).
	\end{proof}
	
	
	Since the localic preimage preserves cozero sublocales, the proof of  Proposition~\ref{F.onto} follows similarly as the proof of Proposition~\ref{ed.onto}. One takes $a,b\in \coz~L$ (resp. $a\in L$ and $b\in \coz~L$) and uses the fact that $L$ is an $F$-frame (resp. basically disconnected frame) and that $f$ is an $\mathfrak{S}_\Sc\beta$-localic map.
	
	\begin{proposition}\label{F.onto}
		Let $f\colon L\to M$ be an dense $\mathfrak{S}_\Sc\beta$-localic map.
		\begin{enumerate}
			\item If $L$ basically disconnected then so is $M$.
			\item If $L$ is an $F$-frame then so is $M$.
		\end{enumerate}
	\end{proposition}

	\subsection{Special cases}
	In this section we take a look at when the notions of disconnectedness and normality coincide. 
	A locale is \emph{Boolean} if every element is complemented; equivalently if every closed and every open sublocale is clopen. On the other hand, a locale $L$ is a \emph{$P$-frame} if every cozero element is complemented (\cite{BW}); again, this is equivalent to saying that every cozero or every zero sublocale is clopen (in a $P$-frame the classes of cozero, zero and clopen sublocales coincide). 
	\begin{proposition}
		$L$ is boolean if and only if it is $\mathfrak{S}_\open$-normally separated if and only if it is $\mathfrak{S}_\close$-disconnected. 
	\end{proposition}
	\begin{proof}
		It is obvious by the definitions that $\mathfrak{S}_\open$-normally separated and $\mathfrak{S}_\close$-disconnected coincide. 
		Now, by Remark~\ref{disc}, $L$ is  $\mathfrak{S}_\close$-disconnected if and only if $\overline{\close(a)}=\close(a)$ is clopen for every $a\in L$. 
	\end{proof}
	\begin{proposition}The following are equivalent for a locale $L$:
		\begin{enumerate}[label=\textup{(\roman*)}]
			\item\label{p1} $L$ is a $P$-frame.
			\item\label{p2} $L$ is $\mathfrak{S}_\Sc$-normally separated.
			\item\label{p3} $L$ is $\mathfrak{S}_\Sz$-disconnected
		\end{enumerate}
	\end{proposition}
	\begin{proof}
		\ref{p1}$\implies$\ref{p2}: Let $C\in \Sc(L)$ and $a\in L$ such that $C\cap \close(a)=\OS$. Since $L$ is a $P$-frame $C$ is also a zero sublocale; thus, $C$ and $C^\#$ are disjoint zero sublocales and $\close(a)\subseteq C^\#$. \\
		\ref{p2}$\implies$\ref{p3}: Let $\close(b)\in \Sz(L)$ and $a\in L$ such that $\close(a)\cap \open (b)=\OS$. We have that, $\open(a)\in \Sc(L)$, $\open (b)\subseteq \open (a)$, and $\close(a)\cap \open (a)=\OS$. By assumption, $\close(b)$ and $\open (b)$ are completeley separated; in particular, so are $\close(b)$ and $\open (a)$, as required.\\
		\ref{p3}$\implies$\ref{p1}: Let $a\in \coz~L$, then $\close(a)\cap \open(a)=\OS$. By assumption, they are completely separated; in particular, we have $\close(a)\cap \overline{\open(a)}=\OS$. Equivalentely, $a\vee a^*=1$. Thus every cozero element of $L$ is complemented. 
	\end{proof}
	
	Then we obtain as corollaries of Theorems~\ref{theorem.b} and \ref{ch.op}:
	
	\begin{corollary}\label{pframe}
		The following are equivalent for a locale $M$:
		\begin{enumerate}[label=\textup{(\roman*)}]
			\item $M$ is a $P$-frame. 
			\item Every closed localic map $f\colon L\to M$ is a $\mathfrak{S}_\Sc\beta$-map.
			\item Every $\z$-closed localic map $f\colon L\to M$  is a $\mathfrak{S}_\Sc\beta$-map.
			\item Every proper localic map $f\colon L\to M$  is a $\mathfrak{S}_\Sc\beta$-map.
			\item  Every open localic map $f\colon L\to M$ is a $\mathfrak{S}_\Sz\beta$-map.
			\item Every $\co$-open localic map $f\colon L\to M$ is a $\mathfrak{S}_\Sz\beta$-map.
		\end{enumerate}
	\end{corollary}
	
	\begin{corollary}\label{char.boolean}The following are equivalent for a locale $M$:
		\begin{enumerate}[label=\textup{(\roman*)}]
			\item $M$ is a boolean. 
			\item Every closed localic map $f\colon L\to M$ is a $\mathfrak{S}_\open\beta$-map.
			\item Every $\z$-closed localic map $f\colon L\to M$  is a $\mathfrak{S}_\open\beta$-map.
			\item Every proper localic map $f\colon L\to M$  is a $\mathfrak{S}_\open\beta$-map.
			\item  Every open localic map $f\colon L\to M$ is a $\mathfrak{S}_\close\beta$-map.
			\item Every $\co$-open localic map $f\colon L\to M$ is a $\mathfrak{S}_\close\beta$-map.
		\end{enumerate}
	\end{corollary}

	\begin{remark}\label{remarkoc}
		Furthermore, since every closed and every open sublocale in a Boolean frame is clopen, a map $f\colon L\to M$ with $M$ boolean is an $\mathfrak{S}_\close\beta$-map if and only if it is a $\mathfrak{S}_\open\beta$-map. 
		Similarly, since zero sublocales and cozero sublocales coincide in a $P$-frame $M$, then $f\colon L\to M$ is an  $\mathfrak{S}_\Sz\beta$-map if an only if it is an $\mathfrak{S}_\Sc\beta$-map.
	\end{remark}
	
	The following proposition is an improvement of \cite[Proposition 4.4]{Socle}, in the sense that we are adding another equivalent condition.
	\begin{proposition}\label{char.boolean2} The following are equivalent for a frame $M$:
		\begin{enumerate}[label=\textup{(\roman*)}]
			\item\label{clo1} $M$ is boolean (resp. $P$-frame)
			\item\label{clo2} Every localic map $f\colon L\to M$ is an $\mathfrak{S}_\open\beta$-map (resp. $\mathfrak{S}_\Sc\beta$-map).
			\item\label{clo3} Every localic map $f\colon L\to M$ is an $\mathfrak{S}_\close\beta$-map. (resp. $\mathfrak{S}_\Sz\beta$-map)			
		\end{enumerate}
	\end{proposition}
	\begin{proof}
		Clearly, by Remark~\ref{remarkoc} \ref{clo2}$\iff$\ref{clo3} holds. By Corollary~ \ref{char.boolean} and Corollary~\ref{pframe}, \ref{clo3}$\implies$\ref{clo1}.
		Let us show \ref{clo1}$\implies$\ref{clo3}. Let $a\in L$ and $b\in M$ (resp. $b\in \coz M$) such that 
		\[\open_L(a)\cap f_{-1}[\close_M(b)]=\OS.\]
		Then $\open_L(a)\subseteq f_{-1}[\open_M(b)]$, and since $M$ is boolean (resp. $P$-frame), $\open_M(b)$ and $\close_M(b)$ are clopen and hence zero sublocales. Thus $\open_M(b)$ and $\close_M(b)$ are completely separated in $M$.
	\end{proof}
	
	To close this section we offer a table showing the different frames that $\mathfrak{S}$-normally separation and $\mathfrak{S}$-disconnectedness describe:
	\begin{table}[h!]
		\begin{center}
			\begin{tabular}{c|c|c}
				\hline
				Selection $\mathfrak{S}$& $\mathfrak{S}$-normally separated & $\mathfrak{S}$-disconnected\\
				\hline
				$\mathfrak{S}_\close $ & normal & Boolean\\
				$\mathfrak{S}_\Sz$&$\delta$-normally separated & $P$-frame\\
				$\mathfrak{S}_\open$&Boolean & Extremally disconnected\\
				$\mathfrak{S}_\Sc $ & $P$-frame & Basically disconnected\\
				$\mathfrak{S}_{\mathsf{reg}\close}$ & Almost normal & Extremally disconnected\\
				$\mathfrak{S}_{\mathsf{reg}\open} $ & Extremally disconnected & Extremally disconnected
			\end{tabular}
		\end{center}
	\end{table}
	
	\section{$\mathfrak{S}\lambda$-maps}
	
	
	Now we will look at the $\mathfrak{S}\gamma$-maps when $\gamma$ is the Lindel\"{o}fication $\lambda$. 
	These maps are a generalization of the $\lambda$-maps introduced in \cite{open,closed}. Indeed, the $\mathfrak{S}_\close\lambda$-maps are precisely the right adjoints of the $\lambda$-maps, and the characterization in \cite[Example 1]{openErratum} comes as a corollary of the equivalence \ref{L1}$\iff$\ref{L4} in the theorem below.
	
	\begin{theorem}\label{Lambdamap}
		Let $f\colon L\to M$ be a localic map, then the following are equivalent:
		\begin{enumerate}[label=\textup{(\roman*)}]
			\item\label{L1} $f$ is a $\mathfrak{S}\lambda$-map.
			\item\label{L2} If $f_{-1}[S]\subseteq Z$ for $S\in\mathfrak{S}M$ and $Z\in \Sz(L)$ then there is $D\in \Sz(M) $ such that $S\subseteq D$ and $f_{-1}[D]\subseteq Z$. 			
			\item\label{L3} If $f_{-1}[S]\cap C=\OS$ for $S\in\mathfrak{S}M$ and $C \in \Sc (L)$ then there is $D\in \Sc(M) $ such that $S\cap D=\OS$ and $C\subseteq f_{-1}[D]$.
			\item\label{L4} If $x\leq \tbigwedge f_{-1}[S]$ for $x\in\coz~L$ and $S\in\mathfrak{S}L$, there exists $y\in \coz~M$ such that $y\leq \tbigwedge S$ and $x\leq h(y)$.
		\end{enumerate}
		Furthermore, if every sublocale in $\mathfrak{S}M$ is complemented the following conditions are also equivalent:
		\begin{enumerate}[label=\textup{(v)}]
			\item\label{44} If $C\subseteq f_{-1}[S^\#]$ for $S\in\mathfrak{S}M$ and $C\in \Sc(L)$ then there is $D\in \Sc(M)$ such that $D\subseteq S^\#$ and $C\subseteq f_{-1}[D]$.
		\end{enumerate}
		\begin{enumerate}[label=\textup{(vi)}]
			\item\label{55} If $Z\vee f_{-1}[S^\#]=L$ for $S\in\mathfrak{S}M$ and $Z\in \Sz(L)$ then there is $Z'\in \Sz(M)$ such that $Z'\vee S^\#=M$ and $f_{-1}[Z']\subseteq Z$.
		\end{enumerate}
	\end{theorem}
	\begin{proof}
		\ref{L1}$\iff$\ref{L2}: By Remark~\ref{remark.g} $f$ is $\mathfrak{S}\lambda$-map if and only if  \[\lambda_L(\tbigwedge f_{-1}[S])\leq h^\lambda \lambda_M(\tbigwedge S).\]
		Since $\lambda_L(\tbigwedge f_{-1}[S])=\left\{b\in \coz L\mid b\leq \tbigwedge f^{-1}[S]\right\}$ and 
		\begin{align*}
			h^\lambda \lambda_M(\tbigwedge S)= \left\{c\in \coz L\mid c\leq h(a)\text{ for some } a\in \coz M \text{ with } a\leq \tbigwedge S\right\},
		\end{align*}
		then $f$ is $\mathfrak{S}\lambda$-map if and only if for every $b\in \coz L $ such that $b\leq \tbigwedge f_{-1}[S]$ there is $a\in \coz M$ with $a\leq \tbigwedge S$ and $b\leq h(a)$. Equivalently, using closed sublocales and \eqref{preimage.close.open}, for every $Z\in \Sz(L)$ such that $\overline{f_{-1}[S]}\subseteq Z$ there is $D\in \Sz(M)$ such that $\overline{S}\subseteq D$ and $f_{-1}[D]\subseteq Z$. Since zero sublocales are closed and a sublocale is contained in a closed sublocale if and only if its closure is, we get \ref{L2}.  \\
		\ref{L2}$\iff$\ref{L3} is obvious. \\
		\ref{L2}$\iff$\ref{L4} is immediate by taking closed sublocales and using equations \eqref{closedsub} and \eqref{preimage.close.open}.\\
		When every $S\in \mathfrak{S}M$ is complemented, using the fact that $f_{-1}[-]$ preserves complements, one can readily see that \ref{L2}$\iff$\ref{L3}$\iff$\ref{44}$\iff$\ref{55}.
	\end{proof}
	
	%

	We say a locale $L$ is \emph{ $\mathfrak{S}$-coz separating} if whenever $\open(a)\cap S=\OS$ for $a\in L$ and $S\in \mathfrak{S}L$, then there is $C\in \Sc(L)$ such that $\open (a)\subseteq C$ and $C\cap S=\OS$. Equivalently, if $S\subseteq \close(a)$ for some $a\in L$ and $S\in \mathfrak{S}L$, then there is $Z\in \Sz(L)$ such that $S\subseteq Z\subseteq \close(a)$.
	
	\begin{remark}
		Clearly, since completely separated sublocales are disjoint, for a selection $\mathfrak{S}$ we have that:
		\begin{itemize}
			\item $\mathfrak{S}\beta$-map $\implies \mathfrak{S}\lambda$-map;
			\item $\mathfrak{S}$-disconnected $\implies \mathfrak{S}$-coz separating. 
		\end{itemize}
	\end{remark}
	
	\begin{example}
		A localic map $f\colon L\to M$ is \emph{$\z$-heavy} if $f_{-1}[\close_M(a)]=\OS$ for some $a\in M$, then there is $Z\in \Sz(M)$ such that $\close_M(a)\subseteq Z$ and $f_{-1}[Z]=\OS$ (these are the right adjoints of $coz$-heavy frame homomorphisms \cite{Mdube}).
		Every $\mathfrak{S}_\close\lambda$-map is $\z$-heavy. Indeed, if $f\colon L\to M$ is a $\mathfrak{S}_\close\lambda$-map and let $a\in M$ such that $f_{-1}[\close_M(a)]=\OS$. By \ref{L2} of \ref{Lambdamap}, there is $Z\in \Sz(M)$ such that $\close_M(a)\subseteq Z$ and $f_{-1}[Z]\subseteq \OS$.
		
		Following classical topology, we say a sublocale $S$ of $L$ is \emph{normally placed} whenever the embedding $j\colon S\hookrightarrow L$ is $\z$-heavy. That is, if $\close_L(a)\cap S=\OS$ for some $a\in L$, then there is $Z\in \Sz(L)$ such that $\close_L(a)\subseteq Z$ and $Z\cap S=\OS$. Equivlently, if $ S\subseteq \open_L(a)$ for some $a\in L$, then there is $C\in \Sc(L)$ such that $S\subseteq C\subseteq \open_L(a)$.
		Obviously, every cozero sublocale is normally placed. Note that every zero sublocale is normally placed if and only if $L$ is $\delta$-normally separated. 
		Moreover, if every open sublocale is coz-heavy then $L$ is perfectly normal. Recall that a locale $L$ is \emph{perfectly normal} if $L=\coz~L$; that is every open (resp. closed) sublocale is a cozero (resp. zero) sublocale. \\
	\end{example}
	
	Let us now look at the analogous result of \ref{ch.op} for $\mathfrak{S}\lambda$-maps:
	
	\begin{theorem}\label{coz.sep}
		Let $\mathfrak{S}$ be a selection of sublocales, and let $M$ be a locale such that every $T\in\mathfrak{S}M$ is complemented. The following are equivalent:
		\begin{enumerate}[label=\textup{(\roman*)}]
			\item\label{l.1} $M$ is $\mathfrak{S}$-coz separating. 
			\item\label{l.2} Every $\co$-open map is a $\mathfrak{S}\lambda$-map.
			\item\label{l.3} Every open map is a $\mathfrak{S}\lambda$-map.
		\end{enumerate}
	\end{theorem}
	\begin{proof}
		\ref{l.1}$\implies$\ref{l.2}: Let $f \colon L\to M$ be a $\co$-open localic map. We will use Theorem~\ref{Lambdamap} to show $f$ is a $\mathfrak{S}\lambda$-map. Let $S\in \mathfrak{S}$ and $C\in \Sc(L)$ such that $C\cap f_{-1}[S]=\OS$, then $C\subseteq f_{-1}[S^\#]$. By the image/preimage adjunction we have $f[C]\subseteq S^\#$; equivalently, since $S$ is complemented, $f[C]\cap  S=\OS$. By assumption, there is $D\in \Sc(M)$ such that $D\cap S=\OS$ and $f[C]\subseteq D$ (equiv. $C\subseteq f_{-1}[D]$), as required.\\
		Since every open map is $\co$-open then clearly \ref{l.2}$\implies$\ref{l.3}.\\
		\ref{l.3}$\implies$\ref{l.1}: Let $a\in M$ and $S\in \mathfrak{S}M$ such that $\open (a)\cap S=\OS$. Consider the open localic embedding $j\colon\open(a)\hookrightarrow M$. By assumption, it is an $\mathfrak{S}\lambda$-map. Furthermore, the sublocale $\open(a)$ is a cozero sublocale of $\open(a)$ and $\open(a)\cap j_{-1}[S]=\open(a)\cap S=\OS$. Thus, there is $D\in \Sc(M)$ such that $D\cap S=\OS$ and $\open (a)\subseteq j_{-1}[D]=\open (a)\cap D$ (that is, $\open(a)\subseteq D$).
	\end{proof}
	
	Let us look at examples of $\mathfrak{S}$-coz separating locales. 
	Recall from \cite{BDGWW} that a locale $L$ is an \emph{$Oz$-frame} if every regular element is a cozero element\footnote{This is not how $Oz$-frames are traditionally defined, but we use one of its characterizations to make it easier for the reader.}, and it is a \emph{weakly $Oz$-frame} if every regular element $x\in L$ is equal to $c^*$ for some $c\in \coz~L$. By taking open sublocales in \cite[Proposition 2.2]{BDGWW} and \cite[Proposition 5.3]{BDGWW} (c.f. \cite[Section 4.3]{tesisA}) one can readily check that indeed $Oz$- and weakly $Oz$-frames coincide with $\mathfrak{S}_\open$- and $\mathfrak{S}_\Sc$-coz separating frames respectively. Thus, we get the following two corollaries of Theorem~\ref{coz.sep}.

	\begin{corollary}
		The following are equivalent for a locale $M$:
		\begin{enumerate}
			\item $M$ is an $Oz$-frame.
			\item Every $\co$-open map is a $\mathfrak{S}_\open\lambda$-map.
			\item Every open map is a $\mathfrak{S}_\open\lambda$-map.
		\end{enumerate}
	\end{corollary}
	
	\begin{corollary}
		The following are equivalent for a locale $M$:
		\begin{enumerate}
			\item $M$ is a weakly $Oz$-frame.
			\item Every $\co$-open map is a $\mathfrak{S}_\Sc\lambda$-map.
			\item Every open map is a $\mathfrak{S}_\Sc\lambda$-map.
		\end{enumerate}
	\end{corollary}

	\subsection{Generalizing $\mathfrak{S}\lambda$-maps}
	In the characterization of $\mathfrak{S}\beta$-maps (Theorem~\ref{theorem.b}\ref{2}), since we are working with complete separation, one can choose to work with open, closed, regular open, regular closed, zero or cozero sublocales (recall Remark~\ref{cs.beta}). This is not the case for $\mathfrak{S}\lambda$-maps. Thus, it is worth introducing more general classes of localic maps.\\
	Let $\mathfrak{S}$ and $\mathfrak{T}$ be two selections of sublocales. 
	We say that a localic map $f\colon L\to M$ is \emph{$\mathfrak{S}\mathfrak{T}\lambda$-map} if whenever $f_{-1}[S]\cap T=\OS$ for $S\in \mathfrak{S}M$ and $T\in \mathfrak{T}L$, there exists $T'\in \mathfrak{T}M$ such that $S\cap T'=\OS$ and $T\subseteq f_{-1}[T']$.
	
	\begin{remark}
		\begin{enumerate}[label=\textup{(\arabic*)}]
			\item Clearly, $\mathfrak{S}\mathfrak{T}_\Sc\lambda$-maps are precisely the $\mathfrak{S}\lambda$-maps.
			\item For $\mathfrak{T}=\mathfrak{T}_\open,\mathfrak{T}_\close,\mathfrak{T}_\Sc,\mathfrak{T}_\Sz,\mathfrak{T}_{\mathsf{reg}\open},\mathfrak{T}_{\mathsf{reg}\close}$ every $\mathfrak{S}\beta$-map is a $\mathfrak{S}\mathfrak{T}\lambda$-map. 
			\item If every $S\in \mathfrak{S}M$ is complemented and $S^\#\in \mathfrak{T}M$, then any map $f\colon L\to M$ is a $\mathfrak{S}\mathfrak{T}\lambda$-map. In particular, if $\mathfrak{S}$ selects complemented sublocales and $(\mathfrak{S}L)^\#\subseteq \mathfrak{T}L$ for every frame $L$ (where $(\mathfrak{S}L)^\#=\{S^\#\mid S\in\mathfrak{S}L\}$), then every localic map is an $\mathfrak{S}\mathfrak{T}\lambda$-map.
			\item For two selections $\mathfrak{S}_1$ and $\mathfrak{S_2}$, if $\mathfrak{S}_2\subseteq \mathfrak{S}_1$ then every $\mathfrak{S}_1\mathfrak{T}\lambda$-map is an $\mathfrak{S}_2\mathfrak{T}\lambda$-map.
			\item Two sublocales are completely separated if they are contained in disjoint zero sublocales; hence, $\mathfrak{S}_\Sz\mathfrak{T}_\Sz\lambda$-maps coincide with $\mathfrak{S}_\Sz\beta$-maps. 
			\item If a localic map $f\colon L\to M$ is $\mathfrak{T}$-preserving ($f[T]\in \mathfrak{T}M$ for every $T\in \mathfrak{T}L$), then it is an $\mathfrak{S}\mathfrak{T}\lambda$-map.
		\end{enumerate}
	\end{remark}
	
	\begin{proposition}
		Let $f\colon L\to M$ be a localic map, then 
		\begin{enumerate}[label=\textup{(\arabic*)}]
			\item\label{no1} $f$ is a $\mathfrak{S}_\open\mathfrak{T}_\open\lambda$-map if and only if $f$ is nearly open. 
			\item\label{no2} $f$ is a $\mathfrak{S}_\Sc\mathfrak{T}_\open\lambda$-map if and only if $f$ is $\co$-nearly open.
			\item\label{no3} $f$ is a $\mathfrak{S}_{\mathsf{reg}\open}\mathfrak{T}_\open\lambda$-map if and only if $f$ is $\ro$-nearly open.
		\end{enumerate}
	\end{proposition}
	\begin{proof}
		Let us first note that a localic map $f\colon L\to M$ is a $\mathfrak{S}_\open\mathfrak{T}_\open\lambda$-map if and only if 
		\begin{equation}\label{nop}
			h(a)\wedge b=0 \implies \exists \ d\in L\quad a\wedge d=0 \text{ and }b\leq h(d)
		\end{equation}
		for every $a\in M$ and $b\in L$.
		Similarly, $f$ is a $\mathfrak{S}_\Sc\mathfrak{T}_\open\lambda$-map (resp. $\mathfrak{S}_{\ro}\mathfrak{T}_\open\lambda$-map) if and only if \eqref{nop} holds for every $a\in \coz~L$ (resp. $a\in L^*$).\\
		\ref{no1}: Let $f\colon L\to M$ be an $\mathfrak{S}_\open\mathfrak{T}_\open$ and let $a\in M$.
		We always have that $h(a)\wedge h(a)^*=0$. By assumption, there is $d\in L$ such that $a\wedge d=0$ and $h(a)^*\leq h(d)$. Putting these two equations together we get that $h(a)^*\leq h(d)\leq h(a^*)$.
		Conversely, assume $f$ is nearly open, and let $a\in M$ and $b\in L$ such that $h(a)\wedge b=0$. Hence, $b\leq h(a)^*=h(a^*)$ and $a\wedge a^*=0$. Thus $f$ is an $\mathfrak{S}_\open\mathfrak{T}_\open\lambda$-map.\\
		Mimicking the argument above for $a\in \coz~M$ and $a\in L^*$ one gets the proof of \ref{no2} and \ref{no3} respectively. 
	\end{proof}
	
	\begin{remark}
		Note that $\mathfrak{S}_\open\mathfrak{T}_{\ro}$-are precisely the maps where 
		\[ h(a)\wedge b^*=0 \implies \exists d\in L\quad a\wedge d^*=0 \text{ and } b^*\leq h(d^*)\]
		for every $a\in M$ and $b\in L$. 
		Hence, if one take $b=h(a)$ one gets that $h(a)^*\leq h(d^*)\leq h(a^*)$, which proves that $\mathfrak{S}_\open\mathfrak{T}_{\ro}$-maps are nearly open. Similarly, one can see that $\mathfrak{S}_\Sc\mathfrak{T}_{\ro}$-maps are $\co$-nearly open, and that $\mathfrak{S}_{\ro}\mathfrak{T}_{\ro}$-maps are $\ro$-nearly open. \\
		Furthermore, assume $f$ is nearly open and $h(a)\wedge b^*=0$ for $a\in M$ and $b\in L$, then we get: 
		\begin{align*}
			h(a)\wedge b^*=0&\iff h(a)\leq b^{**}\iff a\leq f(b^{**})\implies a\leq f(b^{**})^{**}\\
			&\iff a\wedge f(b^{**})^*=0.
		\end{align*}
		Moreover, $b^*\leq h(f(b^{**})^*)$. Indeed, we know that $b^*\leq h(f(b^{**}))^*$ (because $ h(f(b^{**}))\leq b^{**}$) and since $f$ is nearly open $h(f(b^{**}))^*= h(f(b^{**})^*)$. Thus, we have shown $f$ is $\mathfrak{S}_\open\mathfrak{T}_{\ro}\lambda$-map.
	\end{remark}
	
	The following Lemma will allow us to compare some of the localic maps that we have introduced with closed and related localic maps. It is an easy and probably known result among point-free topologists, but no explicit register of it was found in the literature. In \cite[3]{Remote} the authors prove \ref{c1}, but require regularity; in fact, only subfitness in the codomain is needed. Recall from \cite{PP12} that a local is \emph{subfit} if 
	\[a\nleq b\implies \exists c \quad a\vee c=1\neq b\vee c\]
	for every $a,b\in L$.
	\begin{lemma}\label{close}
		Let $M$ be subfit locale and $f\colon L\to M$ a localic map then 
		\begin{enumerate}[label=\textup{(\arabic*)}]
			\item\label{c1} $f$ closed if and only if for every $a\in L$ and $b\in L$ 
			\[f(a\vee h(b))=1 \implies f(a)\vee b=1. \]
			\item\label{c2} $f$ is $\z$-closed if and only if for every $a\in \coz~L$ and $b\in L$ 
			\[f(a\vee h(b))=1 \implies f(a)\vee b=1. \]
			\item\label{c3} $f$ is $\rc$-closed if and only if for every $a\in L ^*$ and $b\in L$ 
			\[f(a\vee h(b))=1 \implies f(a)\vee b=1. \]
		\end{enumerate}
	\end{lemma}
	\begin{proof}
		\ref{c1}: Clearly every closed map satisfies the equation in \ref{c1}. Conversely, note it only suffices to show that $f(a\vee h(b))\leq f(a)\vee b)$ since the other inequality always holds. Suppose it is not true, that is $f(a\vee h(a))\nleq f(a)\vee b$. Since $M$ is subfit there is $c\in M$ such that 
		$f(a\vee h(b))\vee c=1\neq c\vee f(a)\vee b$. 
		Hence, 
		\[1= h(f(a\vee h(b)))\vee  h(c)\leq a\vee h(b)\vee h(c)=a \vee h(b\vee c)\]
		so $1=f(a\vee h(b\vee c))$. By assumption, we get $b\vee c\vee f(a)=1$ contradicting the fact that $1\neq c\vee f(a)\vee b$. Consequently, $f$ is closed.\\	
		Clearly, every $\z$-closed (resp. $\rc$-closed) map satisfies the equation in \ref{c2} (resp. \ref{c3}). The converse follows exactly as the proof of \ref{c1}, but one takes $a\in \coz~L$ (resp. $a\in L^*$).
	\end{proof}
	Note that $\mathfrak{S}_\close\mathfrak{T}_\close$-maps  are precisely those maps where 
	\[a\vee h(b)=1\implies \exists d\in M \quad h(x)\leq a, \  x\vee b=1.\]
	Since $h(f(a))\leq a$, this is equivalent to 
	\[a\vee h(b)=1\implies  f(a)\vee b=1.\]
	From the adjunction between $f$ and $h$ we have that $1\leq f(a\vee h(b))\iff 1\leq a\vee h(b)$. Thus, a map is a $\mathfrak{S}_\close\mathfrak{T}_\close$-map if and only if 
	\[ f(a\vee h(b))=1\implies f(a)\vee b=1\]
	for every $a\in L$ and $b\in M$. Consequently, by Proposition~\ref{close}~\ref{c1}, under subfitness $\mathfrak{S}_\close\mathfrak{T}_\close$-maps and closed localic maps coincide. 
	
	Using equations \eqref{closedsub} and \eqref{preimage.close.open}, one can easily see that a localic map is an $\mathfrak{S}_\close\mathfrak{T}_\Sz$-map (resp. $\mathfrak{S}_\close\mathfrak{T}_{\rc}$-map) if and only if
	\[a\vee h(b)=1 \implies \exists d\in \coz~M \quad b\vee d \ \text{ and } \ h(d)\leq a\]
	\[(resp. \  a\vee h(b)=1 \implies \exists d\in M^* \quad b\vee d \  \text{ and } \ h(d)\leq a) \]
	for every $a\in \coz~L$ (resp. $a\in L^*$) and $b\in M$. Hence, by  Proposition~\ref{close}, the following implications hold for a localic map $f\colon L\to M$ when $M$ is subfit:
	\[ \z\text{-preserving}\implies \mathfrak{S}_\close\mathfrak{T}_\Sz\text{-map}\implies \z\text{-closed}\]
	\[ \rc\text{-preserving}\implies \mathfrak{S}_\close\mathfrak{T}_{\rc}\text{-map}\implies \rc\text{-closed}.\]
	
	\begin{remark}\label{coz.closed}
		A localic map $f\colon L\to M$ is a $\mathfrak{S}_\Sz\mathfrak{T}_\close\lambda$-map if and only if 
		\begin{equation*}
			h(a)\vee b =1\implies \exists d\in L\quad a\vee d=1 \text{ and } h(d)\leq b
		\end{equation*}
		for every $a\in \coz~M$ and $b\in L$. Equivalently, since $f\vdash h$, 
		\begin{equation*}
			f(h(a)\vee b) =1\implies a\vee f(b)=1 
		\end{equation*}
		for every $a\in \coz~M$ and $b\in L$. Assuming complete regularity (see \cite[Lemma 3.2]{closed}), these are precisely the right adjoints of the weakly closed homomorphisms of \cite{closed} (or $coz$-closed in \cite{Remote}). 
		Nevertheless, if one looks closely at the definition, these are not the same as the $\z$-closed maps. The latter generalize the notion of $Z$-maps between topological spaces \cite{Zenor}, which are the continuous maps where images of zero sets are closed. 
		
	\end{remark}
	
	
	Following the definition of $\mathfrak{S}$-coz separating, we say a locale is \emph{$\mathfrak{S}$-$\mathfrak{T}$ separating} if $S\cap \open (a)=\OS$ for some $S\in \mathfrak{S}L$, then ther is $T\in \mathfrak{T}L$ such that $S\cap T=\OS$ and $\open(a)\subseteq T$. 
	
	A localic map $f\colon L\to M$ is a \emph{$\mathfrak{T}$-open} map if $f[T]$ is open for every $T\in \mathfrak{T}L$. For the selections $\mathfrak{T}_\open$,$\mathfrak{T}_{\ro}$ and $\mathfrak{T}_\Sc$, these are the open, $\ro$-open and $\co$-open maps respectively. 
	
	\begin{theorem}\label{lambda1}
		Let $\mathfrak{S}$ and $\mathfrak{T}$ be two selections of sublocales such that $\mathfrak{T}L\subseteq \open (L)$ and $L\in \mathfrak{T}L$ for every locale $L$. Let $M$ be a locale such that every $S\in \mathfrak{S}M$ is complemented, then the following are equivalent:
		\begin{enumerate}[label=\textup{(\roman*)}]
			\item\label{oo1} $M$ is $\mathfrak{S}$-$\mathfrak{T}$ separating.
			\item\label{oo2} Every $\mathfrak{T}$-open localic map $f\colon L\to M$ is an $\mathfrak{S}\mathfrak{T}\lambda$-map.
			\item\label{oo3} Every open localic map $f\colon L\to M$ is an $\mathfrak{S}\mathfrak{T}\lambda$-map.
		\end{enumerate}
	\end{theorem}
	\begin{proof}
		\ref{oo1}$\implies$\ref{oo2}:
		Assume $M$ is  $\mathfrak{S}$-$\mathfrak{T}$-separating, and let  $f\colon L\to M$ be a $\mathfrak{T}$-open localic map. Let $S\in \mathfrak{S}M$ and $T\in \mathfrak{T}L$ such that $T\cap f_{-1}[S]=\OS$. Then, since the localic preimage preserves complementes, $T\subseteq f_{-1}[S]^\#=f_{-1}[S^\#]$. By the image/preimage adjunction, this means $f[T]\subseteq S^\#$. Since $S$ is complemented this is equivalent to $f[T]\cap S =\OS$. By assumption, and since $f[T]$ is open, there is $T'\in \mathfrak{T}M$ such that $f[T]\subseteq T'$ (equiv. $T\subseteq f_{-1}[T']$) and $T'\cap S=\OS$. Thus, $f$ is an $\mathfrak{S}\mathfrak{T}\lambda$-map.\\
		The implication \ref{oo2}$\implies$\ref{oo3} is immediate, since every open map is $\mathfrak{T}$-open (we are assuming $\mathfrak{T}L\subseteq \open (L)$).\\ 
		\ref{oo3}$\implies$\ref{oo1}: Let $a\in M$ and $S\in \mathfrak{S}M$ such that $\open_M(a)\cap S=\OS$. Consider the open localic embedding $j\colon \open_M(a)\hookrightarrow M$. By assumption, $j$ is a $\mathfrak{S}\mathfrak{T}\lambda$-map. Since $\open_M(a)\cap S=\OS=\open_M(a)\cap j_{-1}[S]$ and $\open_M(a)\in \mathfrak{T}\open_M(a)$, there exists $T\in\mathfrak{T}M$ such that $\open_M(a)\subseteq j_{-1}[T]=T\cap \open_M(a)$ (equiv. $\open_M(a)\subseteq T$) and $T\cap S=\OS$. Consequently, $M$ is $\mathfrak{S}$-$\mathfrak{T}$ separating.
	\end{proof}
	
	Similarly, now using closed sublocales, we say that a locale is \emph{$\mathfrak{S}$-$\mathfrak{T}$* separating} if $S\cap \close (a)=\OS$ for some $S\in \mathfrak{S}L$, then there is $T\in \mathfrak{T}L$ such that $S\cap T=\OS$ and $\close(a)\subseteq T$. 
	A localic map $f\colon L\to M$ is a \emph{$\mathfrak{T}$-closed map} if $f[T]$ is closed for every $T\in \mathfrak{T}L$. For the selections $\mathfrak{T}_\close$,$\mathfrak{T}_{\rc}$ and $\mathfrak{T}_\Sz$, these are the closed, $\rc$-closed and $\z$-closed maps respectively. 
	\begin{theorem}\label{lambda2}
		Let $\mathfrak{S}$ and $\mathfrak{T}$ be two selections of sublocales such that $\mathfrak{T}L\subseteq \close (L)$ and $L\in \mathfrak{T}L$ for every locale $L$. Let $M$ be a locale such that every $S\in \mathfrak{S}M$ is complemented, then the following are equivalent:
		\begin{enumerate}[label=\textup{(\roman*)}]
			\item\label{cc1} $M$ is $\mathfrak{S}$-$\mathfrak{T}$* separating.
			\item\label{cc2} Every $\mathfrak{T}$-closed localic map $f\colon L\to M$ is an $\mathfrak{S}\mathfrak{T}\lambda$-map.
			\item\label{cc3} Every closed localic map $f\colon L\to M$ is an $\mathfrak{S}\mathfrak{T}\lambda$-map.
			\item\label{cc4} Every proper localic map $f\colon L$ is an $\mathfrak{S}\mathfrak{T}\lambda$-map.
		\end{enumerate}
	\end{theorem}
	\begin{proof}
		\ref{cc1}$\implies$\ref{cc2}:
		Assume $M$ is  $\mathfrak{S}$-$\mathfrak{T}$* separating, and let  $f\colon L\to M$ be a $\mathfrak{T}$-closed localic map. Let $S\in \mathfrak{S}M$ and $T\in \mathfrak{T}L$ such that $T\cap f_{-1}[S]=\OS$. Then, since the localic preimage preserves complementes, $T\subseteq f_{-1}[S]^\#=f_{-1}[S^\#]$. By the image/preimage adjunction, this means $f[T]\subseteq S^\#$. Since $S$ is complemented this is equivalent to $f[T]\cap S =\OS$. By assumption, and since $f[T]$ is closed, there is $T'\in \mathfrak{T}M$ such that $f[T]\subseteq T' $(equiv. $T\subseteq f_{-1}[T']$) and $T'\cap S=\OS$. Thus, $f$ is an $\mathfrak{S}\mathfrak{T}\lambda$-map.\\
		The implications \ref{cc2}$\implies$\ref{cc3} and \ref{cc3}$\implies$\ref{cc4} are immediate, since every closed map is $\mathfrak{T}$-closed (we are assuming $\mathfrak{T}L\subseteq \close (L)$), and every proper map is closed.\\ 
		\ref{cc4}$\implies$\ref{cc1}: Let $a\in M$ and $S\in \mathfrak{S}M$ such that $\close_M(a)\cap S=\OS$. Consider the proper localic embedding $j\colon \close_M(a)\hookrightarrow M$. By assumption, $j$ is a $\mathfrak{S}\mathfrak{T}\lambda$-map. Since $\close_M(a)\cap S=\OS= \close_M(a)\cap j_{-1}[S]$ and $\close_M(a)\in \mathfrak{T}\close_M(a)$, there exists $T\in\mathfrak{T}M$ such that $\close_M(a)\subseteq j_{-1}[T]=T\cap \close_M(a)$ (equiv. $\close_M(a)\subseteq T$) and $T\cap S=\OS$. Consequently, $M$ is $\mathfrak{S}$-$\mathfrak{T}$* separating.
	\end{proof}
	
	The following two tables show all the $\mathfrak{S}\mathfrak{T}$-separating and $\mathfrak{S}\mathfrak{T}$*-separating that describe our usual selections. Thus, Theorem~\ref{lambda1} and Theorem~\ref{lambda2} offer characterizations for some of the respective frames below.  
	\begin{table}[h!]\label{table.l1}
		\caption {$\mathfrak{S}\mathfrak{T}$-separating} \label{STsep} 
		\begin{center}
			\begin{tabular}{c|c|c|c|c|c|c}
				\hline
				$\mathfrak{S}\diagdown \mathfrak{T}$ &$\mathfrak{T}_\open$&$\mathfrak{T}_\close$&$\mathfrak{T}_\Sc$&$\mathfrak{T}_\Sz$ & $\mathfrak{T}_{\mathsf{reg}\open}$& $\mathfrak{T}_{\mathsf{reg}\close}$ \\
				\hline
				$\mathfrak{S}_\open$ & -- & -- & Oz & Oz & -- & -- \\
				$\mathfrak{S}_\close$&--&Boolean&Perfectly Normal& Boolean &Boolean & Boolean \\
				$\mathfrak{S}_\Sc$&--&--&Weak Oz& --& --&--\\
				$\mathfrak{S}_\Sz$&--&$P$-frame & --& $P$-frame & Almost $P$&$P$-frame\\
				$\mathfrak{S}_{\mathsf{reg}\open}$&--&-- &Oz& Oz& --&--\\
				$\mathfrak{S}_{\mathsf{reg}\close}$&--&ED&Oz&ED&--&ED\\
			\end{tabular}
		\end{center}
		\small
		\hskip -\fontdimen2{Note}:  --  denotes every frame
	\end{table}
	\begin{table}[h!]\label{table.l2}
		\caption {$\mathfrak{S}\mathfrak{T}$*-separating} \label{STsep*} 
		\begin{center}
			\begin{tabular}{c|c|c|c|c|c|c}
				\hline
				$\mathfrak{S}\diagdown \mathfrak{T}$ &$\mathfrak{T}_\open$&$\mathfrak{T}_\close$&$\mathfrak{T}_\Sc$&$\mathfrak{T}_\Sz$ & $\mathfrak{T}_{\mathsf{reg}\open}$& $\mathfrak{T}_{\mathsf{reg}\close}$ \\
				\hline
				$\mathfrak{S}_\open$&Boolean &--&Boolean &Perfectly Normal&Boolean&Boolean\\
				$\mathfrak{S}_\close$&--&--&Normal&Normal&& \begin{minipage}{3cm}
					\tiny{Disjoint closed sublocales are contianed in disjoint regular closed sublocales.}
				\end{minipage}\\
				$\mathfrak{S}_\Sc$&$P$-frame&--&$P$-frame&--&$P$-frame&Almost$P$\\
				$\mathfrak{S}_\Sz$&--&--&--&$\delta$-NS&&\\
				$\mathfrak{S}_{\mathsf{reg}\open}$&ED&--&ED&Oz&ED&--\\
				$\mathfrak{S}_{\mathsf{reg}\close}$&--&--&&&--&\\
			\end{tabular}
		\end{center}
		\small
		\hskip -\fontdimen2{Note}:  --  denotes every frame, and no obvious characterizing class of locales was found  when the cell is left blank. 
	\end{table}
	
	\section{Embeddings}
	
	This section offers some examples and characterizations as to when is a localic embedding an $\mathfrak{S}\beta$- or $\mathfrak{S}\mathfrak{T}\lambda$-map. 
	
	
	In \cite{Mdube}, the authors introduce the notions of $C_1$- and strong $C_1$-quotients. Here we present them in terms of sublocales.
	We say a sublocale is \emph{$C_1$-embedded} if for every $Z\in \Sz(S)$ and $F\in \Sz(L)$ such that $Z\cap (F\cap S)=\OS$, then $Z$ and $F$ are completely separated in $L$ (equiv. the left adjoint of $j\colon S\hookrightarrow L$ is a $C_1$-quotient).  \\
	A sublocale $S$ is \emph{strongly $C_1$-embedded} if for every $Z\in \Sz(S)$ and $a\in L$ such that $\close_L(a)\cap S\in \Sz(S)$ and $Z\cap (\close_L(a)\cap S)=\OS$, then $Z$ and $\close_L(a)$ are completely separated in $L$ (equiv. the left adjoint of $j\colon S\hookrightarrow L$ is a strong $C_1$-quotient). 
	
	It is already shown in \cite[Corollary 3.3]{Mdube} that $j\colon S\hookrightarrow L$ is an $\mathfrak{S}_\Sz\beta$-map if and only if $S$ is $C_1$-embedded. For strong $C_1$-embeddings we have the following result: 
	\begin{proposition}
		Let $S$ be a sublocale of $L$. The localic embedding $j\colon S\hookrightarrow L$ is $\mathfrak{S}_\close\beta$-map if and only if $S$ is strongly $C_1$.
	\end{proposition}
	\begin{proof}
		Let $j$ be an $\mathfrak{S}_\close\beta$-map and let $a\in L$ and $Z\in \Sz(S)$ such that $\close_L(a)\cap S\in \Sz(S)$ and $Z\cap (\close_L(a)\cap S)=\OS$. In particular, $Z$ and $j_{-1}[\close_L(a)]=\close_L(a)\cap S$ are completely separated in $S$ (they are disjoint zero sublocales of $S$). By assumption there is $D\in \Sz(L)$ completely separated from $\close_L(a)$ and $Z\subseteq j_{-1}[D]= D\cap S$. In particular, $Z$ and $\close_L(a)$ are completely separated in $L$. \\
		Conversely, assume $S$ is strongly $C_1$-embedded in $L$. Let $Z\in \Sz(S)$ and $a\in L$ such that $Z$ and $\close_L(a)\cap S$ are completely separated in $S$. Then there is $D\in \Sz(S)$ such that $Z\cap D=\OS$ and $\close_L(a)\cap S\subseteq D$. Since $D$ is a closed sublocale of $S$, there is a closed sublocale $D'$ of $L$ such that $D=D'\cap S$. Then 
		\[(\close_L(a)\vee D')\cap S=(\close_L(a)\cap S)\vee (D'\cap S)=D.\]
		Using the fact tha $S$ is strongly $C_1$-embedded we can conclude $\close_L(a)\vee D'$ and $Z$ are completely separated in $L$. This means there is $Z'\in \Sz(L)$ such that $Z\subseteq Z'$ (in particular, $Z\subseteq Z'\cap S$), and $Z'$ is completely separated from $\close_L(a)\vee D'$; thus, it is also completely separated from $\close_L(a)$.
	\end{proof}

	Recall that a sublocale $S$ of $L$ is \emph{$C^*$-embedded} if every bounded continuous real-valued function on the locale $S$ can be extended to $L$ (see for instance \cite{BW,tesisA,parallel}). For our purposes it is more useful to use one of its well-known characterizations: $S$ is $C^*$-embedded iff whenever two sublocales of $S$ are completely separated in $S$ then they are completely separated in $L$ \cite[Theorem 6.1]{extension}. 
	
	\begin{proposition}\label{strongc1}
		Let $S$ be a sublocale of $L$ and $j\colon S\hookrightarrow L$ its localic embedding, $j$ is an $\mathfrak{S}_\open\beta$-map if and only if  $S$ is $C^*$-embedded and $\overline{S}$ is clopen.
	\end{proposition}
	\begin{proof}
		Suppose $j$ is an $\mathfrak{S}_\open\beta$-map. Immediately, by Remark~\ref{clop.ass}, we obtain that $\overline{S}=\overline{j[S]}$ is clopen. To show $S$ is $C^*$-embedded consider $\close_S(a),\close_S(b)\in\Sz(S)$ such that $\close_S(a)\cap \close_S(b)=\OS$, we will prove they are completely separated in $L$. From Remark~\ref{cs.normal}, there is $\open_S(x)$ such that $\close_S(a)\subseteq \open_S(x)$ and $\open_S(x)$ is completely separated from $\close_S(b)$ in $S$; that is $x\cb_S b$\footnote{To avoid confusion, in this proof we use the subscript on the completely below relation to denote the corresponding locale.}. Since $j$ is $\cb$-preserving (Theorem~\ref{ass}), we get $x\cb_L b$. Meaning $\open_L(x)$ and $\close_L(b)$ are completely separated in $L$. Clearly, $\close_S(b)\subseteq \close_L(b)$ and $\close_S(a)\subseteq \open_S(x)\subseteq \open_L(x)$. Thus, $\close_S(a)$ and $\close_S(b)$ are completely separated in $L$. \\
		Conversely, suppose $S$ is $C^*$-embedded and $\overline{S}$ is clopen. Let $a\in L$ and $D\in \Sz(S)$ such that $D$ is completely separated from $\open_L(a)\cap S=j_{-1}[\open_L(a)]$ in $S$. Since $S$ is $C^*$-embedded, $D$ and $\open_L(a)\cap S$ are completely separated in $L$. Thus, there are $Z_1,Z_2\in \Sz(L)$ such that 
		\[D\subseteq Z_1, \quad \open_L(a)\cap S\subseteq Z_2\quad \text{ and}\quad Z_1\cap Z_2=\OS.\]
		From Remark~\ref{cs.normal}, there is a $C\in \Sc(L)$ such that $Z_1\subseteq C$ and $C$ is completely separated from $Z_2$ in $L$. 
		In particular, we have that $C\cap \open_L(a)\cap S\subseteq C\cap Z_2=\OS$. Hence, $S \subseteq (C\cap \open_L(a))^\#$. Since $(C\cap \open_L(a))^\#$ is a closed sublocale, we observe that 
		$\overline{S} \subseteq (C\cap \open_L(a))^\#$; consequently, $ C\cap \open_L(a)\subseteq \overline{S}^\#$.
		Thus, $\overline{S}\cap C\cap \open_L(a)\subseteq \overline{S}\cap \overline{S}^\#=\OS$, which implies that 
		\[\open_L(a)\subseteq (\overline{S}\cap C)^\#=\overline{S}^\#\vee C^\#.\]
		Note that $\overline{S}^\#\vee C^\#$ and $Z_1\cap\overline{S}$ are zero sublocales (since $\overline{S}$ is clopen); furthermore, they are disjoint:
		\begin{align*}
			(\overline{S}^\#\vee C^\#) \cap (Z_1\cap\overline{S}) &= (\overline{S}^\# \cap Z_1\cap\overline{S})\vee (C^\#\cap Z_1\cap\overline{S})\\
			&\subseteq (\overline{S}^\# \cap\overline{S})\vee (C^\#\cap Z_1)=\OS.
		\end{align*}
		We have then shown that $\open_L(a)$ is completely separated from $Z_1\cap \overline {S}$, and, since $D\subseteq S$ and $D\subseteq Z_1$, we have that  $D\subseteq S\cap Z_1\cap \overline{S}= j_{-1}[\overline{S}\cap Z_1]$, as required. 
	\end{proof}
	
	A sublocale $S$ is \emph{$z$-embedded} if for every $Z\in \Sz(S)$ there is $Z' \in \Sz(L)$ such that $Z=Z'\cap S$ (\cite{extension,GPP-Notes}). Equivalently, if for every $C\in \Sc(S)$ there is $C' \in \Sc(L)$ such that $C=C'\cap S$.
	
	\begin{proposition}\label{zemb}
		Let $S$ be a sublocale of $L$ and $j\colon S\hookrightarrow L$ its localic embedding, then the following are equivalent:
		\begin{enumerate}[label=\textup{(\roman*)}]
			\item\label{z1}$j$ is a $\mathfrak{S}_\open\mathfrak{T}_\Sc\lambda$-map.
			\item\label{z2} $S$ is $z$-embedded and whenever $S\subseteq \close(a)$ for some $a\in L$ there is $D\in \Sc(L)$ such that $S\subseteq D\subseteq \close_L(a)$.
		\end{enumerate}
	\end{proposition}
	\begin{proof}
		\ref{z1}$\implies$\ref{z2}:	We will show $S$ is $z$-embedded using the characterization in \cite[Corollary 2.5.3]{maps}. That is, given disjoint $C_1,C_2\in \Sc(S)$, we will prove there are disjoint $D_1,D_2\in \Sc(L)$ such that $C_i\subseteq D_i\cap S$ for $i=1,2$. Since $C_2$ is a cozero (in particular an open) sublocale of $S$, there is $a\in L$ such that $C_2=\open_L(a)\cap S=j_{-1}[\open_L(a)]$. Hence, $C_1\cap j_{-1}[\open_L(a)]=\OS$. By assumption there is $D_1\in \Sc(L)$ such that $C_1\subseteq j_{-1}[D_1]=D_1\cap S$ and $D_1\cap \open(a)=\OS$. In particular, we have that 
		\[j_{-1}[D_1]\cap j_{-1}[\open_L(a)]=j_{-1}[D_1]\cap C_2=\OS.\]
		Using again the fact that $j$ is a $\mathfrak{S}_\open\mathfrak{T}_\Sc\lambda$-map, there is $D_2\in\Sc(L)$ such that $C_2\subseteq j_{-1}[D_2]=D_2\cap S$ and $D_1\cap D_2=\OS$. 
		Moreover, if $S\subseteq \close_L(a)$ for some $a\in L$, then $S\cap \open_L(a)=S\cap j_{-1}[S]=\OS$. By assumption, since $S$ is  a cozero sublocale of $S$, there is $C\in \Sc(L)$ such that $S\subseteq C$ and $C\cap\open_L(a)=\OS$. Thus, $S\subseteq C\subseteq \close_L(a)$. \\
		\ref{z2}$\implies$\ref{z1}: Let $a\in S$ and $C\in \Sc(L)$ such that $\open_S(a)\cap j_{-1}[C]=\open_L(a)\cap S\cap C=\OS$. Since $S$ is $z$-embedded there is $C'\in \Sc(L)$ such that $C'\cap S=C$. Thus, $\open_L(a)\cap S\cap C'=\open_L(a)\cap S\cap C=\OS$; equivalently, $S$ is contained in the closed sublocale $C'^\# \vee \close_L(a)$. By assumption there is $D\in \Sc(L)$ such that $S\subseteq D\subseteq C'^\#\vee \close_L(a)$. Hence, 
		\[D\cap C'\cap \open_L(a)=\OS\quad\text{and}\quad C=S\cap C'=D\cap C'\cap S=j_{-1}[D\cap C']\]
		and $D\cap C'$ is a cozero sublocale of $L$. 
	\end{proof}

	Let us consider a very special sublocale. Given a locale $L$ the smallest dense sublocale (or \emph{Booleanization}) of $L$ is 
		\[\mathfrak{B}L=\left\{x\in L\mid x=x^{**}\right\}=\left\{x^*\mid x\in L\right\}.\] 
	Let $j\colon \mathfrak{B}L \hookrightarrow L$ be the localic embedding from the booleanization of $L$ to $L$. In \cite[Proposition 4.4]{Socle} it is proved that $j$ is a $\mathfrak{S}_\Sz\beta$-map if and only if $L$ is a $P$-frame. Furthermore, Proposition~\ref{strongc1} together with \cite[Corollary 3.18]{Mdube} show that $j$ is a $\mathfrak{S}_\close\beta$-map if and only if $L$ is a boolean. 
	\begin{proposition}\label{boole}
		Let $L$ be a locale and $j\colon \mathfrak{B}L\hookrightarrow L$. 
		\begin{enumerate}[label=\textup{(\arabic*)}]
			\item\label{B1} $j$ is a $\mathfrak{S}_\open\beta$-map if and only if $L$ is extremally disconnected.
			\item\label{B2} $j$ is a $\mathfrak{S}_\Sc\beta$-map if and only if $L$ is basically disconnected.
		\end{enumerate}
	\end{proposition}
	\begin{proof}
		Let us denote by $h\colon L\to \mathfrak{B}L$ the left adjoint of the localic embedding $j$ given by $h(a)=a^{**}$.\\
		\ref{B1}: Assume $j\colon \mathfrak{B}L\hookrightarrow L$ is an $\mathfrak{S}_\open\beta$-map, let us show $L$ is extremally disconnected. Let $b\in L$, then $\close_L(b^{**})\cap \open_L(b)=\OS$ (because $b\leq b^{**}$). Taking preimages we get 
		\[j_{-1}[\close_L(b^{**})]\cap j_{-1}[\open_L(b)]=\OS\]
		and, since $\mathfrak{B}L$ is boolean, both sublocales above are clopen. Because they are disjoint, they are also completely separated. By assumption, there is $a\in L$ such that
		\begin{equation}\label{star}
			j_{-1}[\close_L(b^{**})]\subseteq j_{-1}[\close_L(a)]
		\end{equation} 
		and $\close_L(a)$ is completely separated from $\open_L(b)$. In fact, by Remark~\ref{cs.closure}, $\close_L(a)$ and $\overline{\open_L(b)}=\close(b^*)$ are completely separated. In particular, they are disjoint closed sublocales, so $a\vee b^*=1$.
		Furthermore, from \eqref{star}, we have that \[\close_{\mathfrak{B}L}(b^{**})=\close_{\mathfrak{B}L}(h(b^{**}))\subseteq \close_{\mathfrak{B}L}(h(a))=\close_{\mathfrak{B}L}(a^{**});\]
		thus, $a^{**}\leq b^{**}$. Since $a\vee b^*=1$, we get 
		$1=a\vee b^*\leq a^{**}\vee b^*\leq b^{**}\vee b^*$. Hence, $b^*$ is complemented.\\
		Conversely, suppose $L$ is extremally disconnected, and let $a^*\in \mathfrak{B}L$ and $b\in L$ such that $\close_{\mathfrak{B}L}(a^*)$ and $j_{-1}[\open_L(b)]$ are completely separated. In particular, $\close_{\mathfrak{B}L}(a^*)\subseteq j_{-1}[\close_L(b)]$. Note that 
		\[j_{-1}[\close_L(b^{**})]=\close_{\mathfrak{B}L}(h(b^{**}))=\close_{\mathfrak{B}L}(b^{**})=\close_{\mathfrak{B}L}(h(b))=j_{-1}[\close_L(b)],\]
		then $\close_{\mathfrak{B}L}(a^*)\subseteq j_{-1}[\close_L(b^{**})]$. Moreover, $\close_L(b^{**})$ and $\open_L(b^{**})$ are disjoint clopen (since $L$ is extremally disconnected) and hence disjoint zero sublcoales; thus, they are completely separated in $L$. Further, since $\open_L(b)\subseteq \open_L(b^{**})$, $\close_L(b^{**})$ and $\open_L(b)$ are completely separated in $L$, as required.\\
		\ref{B2}: The proof follows analogously like in \ref{B1} but we take $b\in \coz L$ and since the $L$ is basically disconnected the argument follows since $b^{**}$ is complemented.
	\end{proof}
	Putting together Proposition~\ref{strongc1} and Proposition~\ref{boole} we get:
	\begin{corollary}
		$L$ is extremally disconnected if and only if $\mathfrak{B}L$ is $C^*$-embedded.
	\end{corollary}   
	
	A further example: since $\mathfrak{B}L$ is dense, by Proposition~\ref{zemb}, $j$ is a $\mathfrak{S}_\open\mathfrak{T}_\Sc\lambda$-map if and only if $L$ is coole (that is, a frame whose booleanization is $z$-embedded \cite{coole}). 
	
	\section*{Acknowledgements}
	The author gratefully acknowledges the financial support from the UWC Research Chair in Mathematics and Applied Mathematics – Topology for Tomorrow. The author also acknowledges the support of a scholarship from the Oppenheimer Memorial Trust (OMT) to support the full-time Postdoctoral studies (OMT Ref. 2023-1907).


\begin{thebibliography}{99}
		
		
		\bibitem{maps} A. B. Avilez, On classes of localic maps defined by their behavior on zero sublocales, {\em Topology and its Applications} 308 (2022) art. no. 107971.
		
		\bibitem{tesisA} A. B. Avilez, {\em Point-Free Study of $z$-Embeddings, More General Classes of Localic Maps, and Uniform Continuity}, PhD thesis, University of Coimbra, Coimbra, Portugal, 2022. 
		
		\bibitem{extension} A. B. Avilez and J. Picado, Continuous extensions of real functions on arbitrary sublocales and $C$-, $C^*$-, and $z$-embeddings, {\em Journal of Pure and Applied Algebra} 225 (2021) art. no. 106702.
		
		\bibitem{BW} R.\thinspace N. Ball and J. Walters-Wayland,
		$C$- and $C^*$-quotients in pointfree topology,
		\emph{Dissertationes Mathematicae \textup(Rozprawy Mat.\textup)} 412 (2002) 1-62.
		
		\bibitem{BG} B. Banaschewski and C. Gilmour, Cozero bases of frames, {\em Journal of Pure and Applied Algebra}
		157 (2001) 1--22.
		
		\bibitem{BDGWW} B. Banaschewski, C. Gilmour, T. Dube and J. Walters-Wayland, Oz in pointfree topology, {\em Quaestiones Mathematicae} 32(2) (2009) 215--227.
		
		\bibitem{BM} B. Banaschewski and C.J. Mulvey, Stone-\v{C}ech compactification of locales I, {\em Houston Journal of Mathematics} 6(3) (1980) 301--312.
		
		\bibitem{BanaPultr} B. Banaschewski and A. Pultr, Variants of Openness, {\em Applied Categorical Structures} 2 (1994) 331--350.
		
		\bibitem{coole} G. Bezhanishvili, F. Dashiell, A. Razafindrakoto and J. Walters-Wayland, Semilattice base hierarchy for frames and its topological ramifications, {\em Applied Categorical Structures}, 32(4) (2024) art. no. 18.
		
		\bibitem{Remote} T. Dube, Remote points and the like in pointfree topology, {\em Acta Mathematica Hungarica} 123 (3) (2009) 203--222.
		
		\bibitem{Socle}	T. Dube, Contracting the socle in rings of continuous functions,
		{\em Rendiconti del Seminario Matematico della Universita di Padova} 123 (2010) 37--54.
		
		\bibitem{Mdube} T. Dube and M. Matlabyane, Concerning some variants of $C$-embedding in pointfree topology, {\em Topology and its Applications}  158 (17) (2011) 2307--2321.
		
		\bibitem{open} T. Dube and I. Naidoo, On openness and surjectivity of lifted frame homomorphisms, {\em Topology and its Applications} 157 (14) (2010)  2159--2171.
		
		\bibitem{openErratum} T. Dube and I. Naidoo, Erratum to ``On openness and surjectivity of lifted frame homomorphisms'', {\em Topology and its Applications} 158 (16) (2011) 2257--2259.
		
		\bibitem{closed} T. Dube and I. Naidoo, When lifted frame homomorphisms are closed, {\em Topology and its Applications} 159 (2012) 3049--3058.
		
		\bibitem{FPP} M.\thinspace J. Ferreira, J. Picado and S. Pinto,  Remainders in pointfree topology,
		{\em Topology and its Applications}  245 (2018) 21-45.
				
		\bibitem{GK} J. Guti\'{e}rrez Garc\'{\i}a and T. Kubiak, General insertion and extension theorems for localic real functions,
		{\em Journal of Pure and Applied Algebra} 215 (2011) 1198--1204.
		
		\bibitem{perfect} J. Gutiérrez García, T. Kubiak and J. Picado, Perfectness in Locales, {\em Quaestiones Mathematicae}  40(4) (2017) 507--518.
		
		\bibitem{parallel} J. Gutiérrez García, J. Picado, On the parallel between normality and extremal disconnectedness, {\em Journal of Pure and Applied Algebra} 218 (2014) 784--803.
		
		\bibitem{GPP-Notes} J. Guti\'{e}rrez Garc\'{\i}a, J. Picado and A. Pultr, Notes on point-free real functions and sublocales, in: {\em Categorical Methods in Algebra and Topology}, Textos de Matem\'{a}tica, DMUC, vol. 46, pp. 167-200, 2014.
		
		
		\bibitem{MV} J. Madden and J. Vermeer, Lindel\"{0}f locales and realcompactness, {\em Mathematical Proceedings of the Cambridge Philosophical Society} 99 (1986) 473--480.
		
		\bibitem{Mthesis} M. Matlabyana, {\em Coz-related and other special quotients in frames}, PhD thesis, University of South Africa, Pretoria, 2012.
		
		\bibitem{PP12}	J. Picado and A. Pultr, \emph{Frames and locales: Topology without points}, Frontiers in Mathematics, vol. 28, Springer, Basel, 2012.
		
		\bibitem{Separation} J. Picado and A. Pultr, \emph{Separation in Point-Free Topology}, Birkh\"auser/Springer, Cham, 2012.
		
		\bibitem{Woods}  R.\thinspace G. Woods, Maps that characterize normality properties and pseudocompactness, {\em Journal of the London Mathematical Society} 7 (1973) 453--461.
		
		\bibitem{Zenor} P. Zenor, A note on $Z$-maps and $WZ$-maps, {\em  Proceedings of the American Mathematical Society} 23 (1969) 273--275.
		
		
	\end{thebibliography}
\end{document}